\documentclass[11 pt]{amsart}
\usepackage{graphicx} 
\usepackage{amsmath}
\usepackage{mathtools}
\usepackage
{amstext, amscd, setspace, tikz-cd,mathtools, enumerate, graphics, latexsym, siunitx}

\usepackage{pst-node}
\usepackage{tikz-cd} 

\usepackage{amssymb}
\usepackage{hyperref}

\usepackage{PTSerif} 

\usepackage[cmtip,all]{xy}
\newcommand{\properideal}{%
\mathrel{\ooalign{$\lneq$\cr\raise.22ex\hbox{$\lhd$}\cr}}}

\newcommand{\properring}{%
\mathrel{\ooalign{$\gneq$\cr\raise.22ex\hbox{$\rhd$}\cr}}}

\newcommand{\Z}{{\mathbb Z}}

\newtheorem{theorem}{Theorem}[section]
\newtheorem{corollary}{Corollary}[theorem]

\newtheorem{lemma}[theorem]{Lemma}
\newtheorem{proposition}[theorem]{Proposition}

\newtheorem{definition}[theorem]{Definition}
\newtheorem{example}[theorem]{Example}

\newtheorem{problem}[theorem]{Problem}

\newtheorem{remark}{Remark}[theorem]

\numberwithin{equation}{section}

\catcode`\@=11
\catcode`\@=12

\theoremstyle{definition}
\numberwithin{equation}{subsubsection}

\title[Root Clusters over Number fields : Inverse Problems and Applications]{Root Clusters over Number fields: Inverse Problems and Applications}

\author{Shubham Jaiswal}

\address{Department of Mathematics IIT Bombay, Powai, Mumbai 400 076, India.}

\email{sjaiswal@math.iitb.ac.in} 

\date{February 27, 2026.}

\begin{document}

\begin{abstract}
 We develop the theory of root clusters further in this article and give some applications. We introduce some new notions as well as recall earlier notions for field extensions over a perfect base field: root cluster size, its generalization root capacity, its dual
notion ascending index and its generalization intersection indicium, and generalization of degree of extension, compositum indicium. We establish our results on the Inverse problems
for these generalized notions over number fields which generalizes our earlier results. We give a field theoretic formulation for the concept of minimal generating sets of splitting fields of polynomials which was introduced by the author and Vanchinathan. We present new results as well as generalizations of our earlier results on the cardinalities of
minimal generating sets for extensions over number fields. We generalize a result of Drungilas et al. by establishing that a certain family of triplets is compositum feasible over any number field and we also list all the irreducible triplets in this family. We also prove a partial case of a conjecture of Drungilas et al. Our methods for all these problems are Galois theoretic in nature and heavily
rely on the known cases of the inverse Galois problem.

\end{abstract}

\maketitle

\section{Introduction}
The theory of root clusters was substantially developed by the author and Bhagwat in their work in \cite{Bhagwat_2025} which builds on previous work by Perlis in \cite{perlis2004roots} and Krithika and Vanchinathan in \cite{krithika2023root}. This article is yet another contribution in enriching the theory further by means of systematic generalizations to more expansive frameworks of root capacity and intersection indicium and compositum indicium that provide metrics that allow
for a more nuanced measurement of how roots of irreducible polynomials are distributed within various field extensions. We also apply similar methods to resolve certain problems on minimal generating sets and compositum feasible triplets. 

\subsection{Inverse Problems in Root Cluster Theory}

Let $K$ be a perfect field. We fix an algebraic closure $\bar{K}$ once and for all, and work with finite extensions of $K$ contained in $\bar K$. Let $L/K$ be a degree $n$ extension and let $\tilde{L}$ be its Galois closure inside $\bar{K}$. Let $G = {{\rm Gal}}(\tilde{L}/K)$ and $H= {{\rm Gal}}(\tilde{L}/L)$. We have the notion of cluster size of $L/K$, $r_K(L)$ which is $[N_G(H):H]$ (See Section 2.1 in \cite{Bhagwat_2025} for basic properties of cluster size of field extension). From Section 3.2 in \cite{Bhagwat_2025}, number of clusters of $L/K$, $s_K(L)$ is $[G:N_G(H)]$ which is also the number of distinct fields inside $\bar{K}$ isomorphic to $L$ over $K$. The following is referred to as the Inverse cluster size problem for number fields in \cite{Bhagwat_2025}.

\begin{theorem}[Theorem 3.1.1, \cite{Bhagwat_2025}]
    \label{n,r} 
   Let $K$ be a number field. Let $n>2$ and $r|n$. Then there exists an extension $L/K$ of degree $n$ with cluster size $r$. 
\end{theorem}

In Section \ref{unique inter ext}, we prove some interesting properties of unique intermediate extensions for given extensions. This notion was introduced in Section 7 of \cite{Bhagwat_2025}. The concepts of strong cluster magnification and root capacity were introduced by the author and Bhagwat in Sections 4 and 6 of \cite{Bhagwat_2025} respectively. 

\begin{definition}[Definition 6.2.1 \cite{Bhagwat_2025}]
    \label{root capacity}
Let $L/K$ be an extension. By primitive element theorem $L=K(\alpha)$ for some $\alpha\in\bar{K}$. Let $f$ be minimal polynomial of $\alpha$ over $K$. \smallskip

For an extension $M/K$, root capacity of $M$ with respect to $L$ (with base field $K$ fixed) $\rho_K(M,L)$ is the number of roots of $f$ that are contained in $M$. (This is well defined by Proposition 6.2.2 in \cite{Bhagwat_2025}).  \smallskip

Equivalently by Proposition 6.2.6 (1) in \cite{Bhagwat_2025}, $\rho_K(M,L)=a\ r_K(L)$ where $a$ is number of distinct fields inside $M$ isomorphic to $L$ over $K$.
\end{definition}

In Section \ref{Root Capacity and Cluster Towers}, we establish the Root Capacity Magnification result, Proposition \ref{root capacity mag} which is a generalization of Cluster Magnification Theorem proved in \cite{krithika2023root}. The concept of cluster towers was introduced in \cite{krithika2023root}. We give a field theoretic formulation for cluster towers and prove Proposition \ref{SCM cluster tower} about strong cluster magnification and cluster towers. In Section \ref{Inverse Root Capacity Problem}, we establish the following as a generalization of Theorem \ref{n,r}. 

\begin{theorem}  
\label{inv root cap}

(Inverse Root Capacity Problem for Number Fields) Let $K$ be a number field. Given $(n,r,\rho)$ where $n>2$ and $r|n$ and $r|\rho$ and $\rho\leq n$ and $\rho\neq n-1$. There exist extensions $L/K$ and $M/K$ such that $[L:K]=n$ and $r_K(L)=r$ and $\rho_K (M,L)=\rho$. For $\rho\neq 0$, we get $M/K$ as an extension of $L/K$ contained in $\tilde{L}$.\end{theorem}

The concept of ascending index of an extension was introduced by the author and Bhagwat in Section 7 of \cite{Bhagwat_2025}. For an extension $L/K$, the ascending index of $L/K$, $t_K(L)$ is $[G:H^G]$ and the quantity $u_K(L)$ is $[H^G:H]$ where $H^G$ is the normal closure of $H$ in $G$, i.e. the intersection of all normal subgroups of $G$ that contain $H$ (See Section 7.2 and Section 9 in \cite{Bhagwat_2025} for basic properties of ascending index of field extension). The following is referred to as the Inverse ascending index problem for number fields in \cite{Bhagwat_2025}.

\begin{theorem}[Theorem 9.0.5, \cite{Bhagwat_2025}]
    \label{t,r}
    Let $K$ be a number field. Let $n>2$ and $t|n$. Then there exists an extension $L/K$ of degree $n$ with ascending index $t_K(L)=t$. 
\end{theorem}

Just like we introduced root capacity as a generalization of cluster size in \cite{Bhagwat_2025}, in Section \ref{intersection indicium} of this article we introduce the concept of intersection indicium as a generalization of ascending index. We establish the Intersection Indicium Magnification result, Proposition \ref{tau magn} which is a generalization of Ascending Index Magnification Theorem proved in \cite{Bhagwat_2025}. In Section \ref{inv tau section}, we establish the following as a generalization of Theorem \ref{t,r}. For notations see Sections \ref{intersection indicium} and \ref{inv tau section}.

\begin{theorem}\label{inv tau}
 (Inverse Intersection Indicium Problem for Number Fields) Let $K$ be a number field. Given $(n,t,\tau)$ where $n>2$ and $t\ |\ \tau\ |\ n$ (except for the cases listed below in Remark \ref{excluded}). There exist extensions $L/K$ and $M/K$ such that $[L:K]=n$ and $t_K(L)=t$ and intersection indicium $\tau_K (M,L)=\tau$. We get $M/K$ as an extension of $L/K$ contained in $\tilde{L}$.

\end{theorem}

\begin{remark}
\label{excluded}
    The following cases for $(n,t,\tau)$ are not covered by our proof.\begin{enumerate}
        \item $(n,1,2)$ for $n>2$ and $2|n$.

  \item   $(2\tau, 1, \tau)$ for $\tau>1$.

   \item  $(4t, t, 2t)$ for $t\geq 1$ odd.

   \end{enumerate}

   \end{remark}

   Note that the subcase $(4,1,2)$ is common to all three classes in Remark \ref{excluded} and Proposition \ref{excep prop} proves that this is an impossible case.\smallskip

In Section \ref{compositum indicium}, we introduce the concept of compositum indicium as a generalization of degree of extension. We establish the Compositum Indicium Magnification result, Proposition \ref{gamma magn}. In Proposition \ref{tau interesting eg}, we compute root capacity and intersection indicium and compositum indicium for an interesting example. In Section \ref{inv gamma section}, we state the Inverse Compositum Indicium Problem for number fields and go on to establish some interesting cases of the problem which is encapsulated in the following result. For notations see Sections \ref{compositum indicium} and \ref{inv gamma section}.

\begin{theorem}

Let $K$ be a number field. For the following pairs $(n,\gamma)$ where $n>2$ and $n \mid \gamma\mid  n!$, there exist extensions $L/K$ and $M/K$ such that $[L:K]=n$ and compositum indicium $\gamma_K (M,L)=\gamma$.

\begin{enumerate}
\item $(n,\gamma)$ where $n=2^m a_1a_2\cdots a_k$ with each $a_i>2$ and $m=0$ or $1$ and $\gamma=2^m b_1b_2\cdots b_k$ with each (i) $b_i=\ ^{a_i}P_j$ for $j\leq a_i$ or (ii) $b_i=a_i\phi(a_i/l)$ for $a_i$ odd and $l\mid a_i$ (where $\phi$ is the Euler totient function) or (iii) $b_i=a_i\cdot r^{a-1}$ for $r>1$ and $r\mid a_i$ and $a\leq (a_i/r)$.

\smallskip

\item $(n,\gamma)$ so that for each prime $p$, there exists $k_p\leq v_p(n)$ such that $k_p\mid (v_p(\gamma)-v_p(n))$ and $v_p(\gamma)\leq p^{(v_p(n)-k_p)}\cdot k_p+ (v_p(n)-k_p)$  (where $v_p$ is usual $p$-adic valuation).

\end{enumerate}
\end{theorem}

\subsection{Applications of Root Cluster Theory}

The notion of minimal generating sets of the splitting field of a polynomial was introduced by the author and Vanchinathan in Section 2 in \cite{jaiswal2025minimal}. In Section \ref{Minimal Generating Sets of Galois Closure}, we give a field theoretic formulation for minimal generating sets and prove Proposition \ref{SCM min gen} about strong cluster magnification and minimal generating sets. We prove the following as a significant generalization of the Main Theorem in \cite{jaiswal2025minimal}.

\begin{theorem}\label{inv min gen xn}

Let $K$ be a number field and $n>2$ be an integer.

  \begin{enumerate}

\item  Suppose $\frac{n}{2^{v_2(n)}}>2$ is composite (where $v_2$ is usual $2$-adic valuation). Then there exists an $L/K$ of degree $n$ for which the Galois closure has minimal generating sets of cardinalities $2,3,\dots , \omega'(n)$ (where $\omega'(n)$ is the number of distinct odd prime divisors of $n$) and these are the only possible cardinalities for minimal generating set for $L/K$.

\smallskip

\item  For any $n>2$ and $d\mid n$ with $d>2$, there exists an $L/K$ of degree $n$ for which the Galois closure has all its minimal generating sets of cardinality $k$ (and all cluster towers of length $k$) for the following values of $k$ : (i) $k=d-1$, (ii) $k=d-2$, (iii) $k=2$.

\smallskip

\item For $K=\mathbb{Q}$, there exists an $L/\mathbb{Q}$ of degree $n$ for which the Galois closure has all its minimal generating sets of cardinality $k$ (and all cluster towers of length $k$) for the following values of $n$ and $k$:

\begin{enumerate}

\item (i) $n$ is a multiple of 12 and $k=5$, (ii) $n$ is a multiple of 11 and $k=4$.

\item $n$ is a multiple of $p+1$ where $p$ is an odd prime and $k=3$.

\end{enumerate}
\end{enumerate}

    
\end{theorem}

We then go on to establish the following interesting result (by applying similar methods as in earlier sections). For notations see Section \ref{Minimal Generating Sets of Galois Closure}. 

\begin{theorem}\label{inverse min gen}
    Let $K$ be a number field. Given positive integers $n>2$ and $s|n$ with $s<n$. There exists an $L/K$ of degree $n$ for which the Galois closure has a minimal generating set of cardinality $s$.\smallskip
    
    Furthermore that $L/K$ satisfies $s_K(L)=s$. Hence there is a unique minimal generating set for the Galois closure of $L/K$ which is thus, also a shortest minimal generating set. 
\end{theorem}

 Drungilas et al. introduced the following notion in \cite{drungilas2012degree} with $\mathbb{Q}$ as the base field: A triplet of positive integers $(a,b,c)$ is said to be compositum feasible over $\mathbb{Q}$ if there exist $L/\mathbb{Q}$ and $L'/\mathbb{Q}$ of degrees $a$ and $b$ respectively with the compositum $LL'/\mathbb{Q}$ having degree $c$. 

\begin{theorem}[Theorem 7, \cite{drungilas2012degree}]
\label{drungilas thm}

    A triplet $(a, b, c)$ satisfying $max\{v_p(a), v_p(b)\} \leq v_p(c) \leq v_p(a) + v_p(b)$
for every prime number $p$ (where $v_p$ is usual $p$-adic valuation) is compositum feasible over $\mathbb{Q}$.
\end{theorem}

In our discussion on compositum feasible triplets, we take base field $K$ to be any perfect field unless specified otherwise. 

\begin{definition}
Let $K$ be a perfect field. A triplet of positive integers $(a,b,c)$ is said to be compositum feasible over $K$ if there exist $L/K$ and $L'/K$ of degrees $a$ and $b$ respectively with the compositum $LL'/K$ having degree $c$. We denote the set of compositum feasible triplets over $K$ by $\mathcal{C}_K$.
\end{definition}

Observe that a triplet $(a, b, c)$ satisfies $max\{v_p(a), v_p(b)\} \leq v_p(c) \leq v_p(a) + v_p(b)$
for every prime number $p$ (where $v_p$ is usual $p$-adic valuation) $\iff$ $c=lcm(a,b)\cdot t$ where $t\ |\ gcd(a,b)$ $\iff$  $lcm(a,b)\ |\ c\ |\ ab$. Let $\mathcal{C}'_K$ be the set consisting of triplets of the form $(a,b,c)$ with $c=lcm(a,b)\cdot t$ where $t\ |\ gcd(a,b)$. We establish Theorem \ref{comp feas thm} (1) below (by applying similar methods as in earlier sections) as a generalization of Theorem \ref{drungilas thm} from rationals to number fields. Also Theorem \ref{comp feas thm} (2), (3) is a partial case of Conjecture 4 in \cite{drungilas2012degree} where it is conjectured that $\mathcal{C}_K$ is a commutative multiplicative monoid with $(1,1,1)$ as the identity element. The notion of a compositum feasible triplet being irreducible was introduced in \cite{maciulevivcius2023degree}. In Theorem \ref{comp feas thm} (4), we list all irreducible triplets in $\mathcal{C}'_K$. Propositions \ref{Sn comp feas} and \ref{xn comp feas} cover some other interesting cases of compositum feasible triplets and sum feasible triplets.

\begin{theorem}\label{comp feas thm}

Let $K$ be a number field. 
   \begin{enumerate}

   \item $\mathcal{C}'_K\subset \mathcal{C}_K$. We can choose the field extensions such that compositum has Galois closure over $K$ having a solvable Galois group. We also have that $(a,b,lcm(a,b)\cdot t)$ is a sum feasible and product feasible triplet over $K$.
   
  \smallskip

   \item Suppose $(a,b,c), (a',b',c')\in \mathcal{C}_K$. If $(a,b,c)\in \mathcal{C}'_K$, then we have that $(aa',bb',cc')\in \mathcal{C}_K$. 

\smallskip

   \item $\mathcal{C}'_K$ is a commutative multiplicative monoid. 

\smallskip

\item  A triplet in $\mathcal{C}'_K$ is irreducible if and only if it is one of the following where $p$ is any prime:\\
(i) $(1,1,1)$ (ii) $(1,p,p)$ (iii) $(p,1,p)$ (iv) $(p,p,p)$.

\end{enumerate}
    \smallskip

\end{theorem}


We also improve on the inverse problems proved in \cite{Bhagwat_2025} and this article by proving that there exist arbitrarily large finite families of pairwise non-isomorphic extensions having additional properties that satisfy the given conditions. The following is an improvement on Theorem \ref{n,r}. We obtain similar improvements for Theorems \ref{t,r}, \ref{inv root cap},  \ref{inv min gen xn} (1),  \ref{inverse min gen} and \ref{comp feas thm} (1).

 \begin{theorem}
     
\label{inverse cluster size improv}

 Let $K$ be a number field. Let $n>2$ and $r|n$. Then we get arbitrarily large finite families of extensions $L/K$  inside $\bar{K}$ which are pairwise non-isomorphic over $K$ and are pairwise linearly disjoint over $K$ and each having degree $n$ with cluster size $r_K(L)=r$. 
 \end{theorem}


\section{Inverse Problems in Root Cluster Theory}

 \subsection{Some Remarks on Unique Intermediate Extensions}\label{unique inter ext}

Let $N/K$ be the unique intermediate extension of $L/K$ such that $L/N$ is Galois with maximum possible degree as in Section 7.1 of \cite{Bhagwat_2025}.

 \begin{proposition}
   Consider $L/P/K$.

   \begin{enumerate}
       \item 
  Then $r_K(L)/r_P(L)=[N_G(H):N_{G_0}(H)]$ where $G_0=Gal(\tilde{L}/P)$. In particular $r_P(L)|r_K(L)$.\smallskip

  \item  Let unique intermediate extension for $L/P$ be $N_1/P$. Then $NP= N_1$.

  \item $s_K(L)\ |\ (s_P(L)[P:K])$.

   \end{enumerate}

  \end{proposition}

   \begin{proof} We will prove (1) and (2).

   \begin{enumerate}
       \item Now $\tilde{L}/P$ is Galois with Galois group $G_0$. From the first proposition in \cite{perlisroots}, $r_P(L)=|Aut(L/P)|=[N_{G_0}(H):H]$. Thus $r_K(L)/r_P(L)=[N_G(H):H]/[N_{G_0}(H):H]=[N_G(H):N_{G_0}(H)]$. Since $N_{G_0}(H)\subset N_G(H)$, we have $r_P(L)|r_K(L)$.\smallskip

       \item By Theorem 7.1.1 in \cite{Bhagwat_2025}, $N=\tilde{L}^{N_G(H)}$ and $N_1=\tilde{L}^{N_{G_0}(H)}$. Also $N_G(H)\cap G_0=N_{G_0}(H)$. Thus $NP= N_1$.\end{enumerate}\end{proof}

Let $\sigma_i$ be coset representatives of $N_G(H)$ in $G$ with $\sigma_1=1$. Let $H_i=\sigma_i H \sigma_i^{-1}$. Then $L_i=\tilde{L}^{H_i}$ are the $s_K(L)$ many distinct fields isomorphic to $L$ over $K$.\smallskip

Let $F/K$ be the unique intermediate extension of $L/K$ which is Galois with maximum possible degree as in Section 7.2 of \cite{Bhagwat_2025}.

\begin{proposition}\hfill
\label{unique F}
    \begin{enumerate}
        \item $F=\cap_{i=1}^{s} L_i$ where $s=s_K(L)$.\smallskip

        \item Let $L=K(\alpha)$ for a primitive element $\alpha\in \bar{K}$ with minimal polynomial $f$ over $K$. A cluster of $f$ is defined as the subset of roots of $f$ belonging to the field generated by a single root of $f$ over $K$ and clusters form equivalence classes for the obvious equivalence relation among the roots of $f$ (For details of this notion, see \cite{krithika2023root} and \cite{perlis2004roots}). Let $\{\beta_i\}_{=1}^s$ be a complete set of representatives of root clusters of $f$. Then $F=\cap_{i=1}^{s}  K(\beta_i)$.\smallskip  

        \item $F/K$ is the unique intermediate extension of each $L_i/K$ which is Galois with maximum possible degree.
        
    \end{enumerate}
\end{proposition}

\begin{proof}\hfill

\begin{enumerate}
    \item  We have from Theorem 7.2.1 in \cite{Bhagwat_2025}, that $F=\tilde{L}^{H^G}$. We observe that $H^G$ is the subgroup of $G$ generated by all $H_i$'s for $1\leq i\leq s$. Hence by Galois correspondence, $F=\cap_{i=1}^{s} L_i$. 

    \item Now number of clusters of $f$ is $s_K(L)=s$. By a suitable reordering, we have that each $L_i=K(\beta_i)$.

    \item Follows from part (1).
\end{enumerate}\end{proof}

\begin{remark}
    Proposition \ref{unique F} helps us to give alternate proofs for certain cases in Section 7.3 in \cite{Bhagwat_2025}. The following example encapsulates those cases with the alternate proofs.
    \end{remark}

    \begin{example}\hfill\label{unique F example}
        \begin{enumerate}

        \item Consider $K=\mathbb{Q}$ and let $n>2$. Fix $\zeta$ to be a primitive $n$-th root 
of unity in $\bar{\mathbb{Q}}$. Let $c$ be a positive rational number such that $f=x^n-c$ is
  an irreducible polynomial over $\mathbb{Q}$. Let $a = c^{1/n}$ be the positive real root of $f$. Assume $n$ to be either odd; or even with $\sqrt{c}\not \in \mathbb{Q}(\zeta)$ (Similar to conditions in Example 5.1.3 in \cite{Bhagwat_2025}). Then the unique $F/\mathbb{Q}$ for $\mathbb{Q}(a)/\mathbb{Q}$ is $\mathbb{Q}$ for $n$ odd \& $\mathbb{Q}(a^{n/2})$ for $n$ even.  
  
  By Proposition 1 in \cite{jacobson1990galois} and Theorem A in \cite{jacobson1990galois}, $n$ is odd or, $n$ is even with $\sqrt{c}\not \in \mathbb{Q}(\zeta)$  if and only if  $\mathbb{Q}(a)\cap \mathbb{Q}(\zeta)=\mathbb{Q}$. Thus $[\mathbb{Q}(\zeta)(a):\mathbb{Q}(\zeta)]=n$. Hence we have the set $\{a^i\}_{i=0}^{n-1}$ to be linearly independent over $\mathbb{Q}(\zeta)$. Let $\gamma \in \mathbb{Q}(a)\cap \mathbb{Q}(a\zeta)$.\\ 
        Then $\gamma=a_0 +a_1 a + \dots + a_{n-1} a^{n-1}= b_0 +b_1 (a\zeta) + \dots + b_{n-1} (a\zeta)^{n-1}$ for $a_i,b_i\in \mathbb{Q}$ for all $0\leq i\leq n-1$. Thus $a_0 +a_1 a + \dots + a_{n-1} a^{n-1}= b_0 +(b_1 \zeta) a + \dots + (b_{n-1} \zeta^{n-1})a^{n-1}$. Hence for all $0\leq i\leq n-1$, we have $a_i=b_i \zeta^i$. If $n$ is odd, $a_i=0$ for all $i\neq 0$ and if $n$ is even, $a_i=0$ for all $i\neq 0, n/2$. Also for $n$ even, $a^{n/2}\in \mathbb{Q}(a\zeta^i)$ for all $0\leq i\leq n-1$. Hence by Proposition \ref{unique F}, for n odd, $F=\mathbb{Q}$ and for n even $F=\mathbb{Q}(a^{n/2})$.
\bigskip

            \item Let $f$ over $K$ be irreducible of degree $n>2$ with Galois group ${\mathfrak S}_n$ with roots $\alpha_i\in \bar{K}$ for $1\leq i\leq n$. Let $L=K(\alpha_1)$. Then the unique $F/K$ for $L/K$ is $K$.\medskip

 Since the Galois group is ${\mathfrak S}_n$, we have $[K(\alpha_1,\alpha_2): K(\alpha_1)]=n-1$. Hence $K(\alpha_1)$ and $K(\alpha_2)$ are distinct fields. Let $[F:K]=t$. By Proposition \ref{unique F}, $F\subset K(\alpha_2)$. We have $[K(\alpha_1):F]=[K(\alpha_2):F]=n/t$. Since $[K(\alpha_1,\alpha_2): K(\alpha_1)]\leq [K(\alpha_2):F]$. Thus $n-1\leq n/t$. Which holds only if $t=1$ i.e. $F=K$.
        \end{enumerate}
    \end{example}

 \begin{proposition}
   Consider $L/P/K$.

   \begin{enumerate}
       \item Then $t_K(L)/t_K(P)=[G_0^G:H^G]$ where $G_0=Gal(\tilde{L}/P)$. In particular $t_K(P)|t_K(L)$.

       \item Let unique intermediate extension for $L/P$ be $F_1/P$. Then $F\subset F_1$. Thus $t_P(L)[P:K]=t_K(L)[H^G: H^{G_0}]$. In particular $t_K(L)\ |\ (t_P(L)[P:K])$.

       \item $u_P(L)|u_K(L)$ and $u_K(L)\ |\ (u_K(P)[L:P])$.

       \end{enumerate}

  \end{proposition}

  \begin{proof}
We will prove (2). Since $H^G\cap G_0\unlhd G_0$. Hence $H^{G_0}\subset H^G\cap G_0\subset H^G$. From Theorem 7.2.1 in \cite{Bhagwat_2025}, $F=\tilde{L}^{H^G}$ and $F_1=\tilde{L}^{H^{G_0}}$. Thus $F\subset F_1$. Now $t_P(L)=[G_0:H^{G_0}]=[F_1:P]$. Since $t_K(L)=[G:H^G]=[F:K]$ and $[F:K][F_1:F]=[F_1:K]=[F_1:P][P:K]$, we are done.
  \end{proof}

\subsection{Root Capacity and Cluster Towers}\label{Root Capacity and Cluster Towers}

We begin with the following definition.

\begin{definition}[Definition 4.1.1 \cite{Bhagwat_2025}]
\label{SCM}
A finite extension $M/K$ is said to be obtained by strong cluster magnification from a subextension $L/K$ if we have the following:
\smallskip
 
 \begin{enumerate}
 \item  $[L:K] = n > 2,$ 
 \smallskip
      
\item there exists a finite Galois extension $F/K$ such that the Galois closure $\tilde{L}$ of $L/K$ in $\bar{K}$ and $F$ are linearly disjoint over $K$ i.e. $\tilde{L}\cap F=K$. \smallskip

\item $LF=M$.\smallskip

 \end{enumerate}

The number $[F:K]$ is called the magnification factor and denoted by $d$.

\end{definition}

   The following example negatively answers the Problem 10.2.6 on strong cluster magnification in Chapter 10 of PhD Thesis \cite{ShubhamJaiswalPhDThesis} of the author.

\begin{example}

 Consider the case in Example \ref{unique F example} (1) with $n\equiv 2\ (\text{mod}\ 4)$. From Example 7.3.3 in \cite{Bhagwat_2025}, we have that $\mathbb{Q}(a)/\mathbb{Q}$ is obtained by nontrivial strong cluster magnification from $\mathbb{Q}(a^{2})/\mathbb{Q}$ through $\mathbb{Q}(a^{n/2})/\mathbb{Q}$. We also have that both the extensions $\mathbb{Q}(a^{2})/\mathbb{Q}$ and $\mathbb{Q}(a)/\mathbb{Q}(a^{n/2})$ have cluster size 1.\smallskip

  Hence we conclude that if $M/K$ is obtained by nontrivial strong cluster magnification from some $L/K$ and $K\subset M'\subset M$ and $K\subset K'\subset M$. Then it is not necessary that $M'/K$ or $M/K'$ are obtained by nontrivial strong cluster magnification from some subextensions.

\end{example}

\smallskip

The following is a reformulation (in terms of strong cluster magnification property) of a result proved for polynomials in Section 3.1 of \cite{krithika2023root} and reformulated for field extensions in Section 4 of \cite{krithika2023root} which is referred to as the Cluster Magnification theorem.  

\begin{proposition}

\label{cluster magnification}
    Let $M/K$ be obtained by strong cluster magnification from $L/K$ with magnification factor $d$. Then $[M:K]=d\ [L:K]$ and $r_K(M)=d \ r_K(L)$.
\end{proposition}

\begin{remark}
Suppose $M/K$ and $M'/K$ are isomorphic over $K$ and $L/K$ and $L'/K$ are isomorphic over $K$. Then $\rho_K(M,L)=\rho_K(M',L')$. 
\end{remark}

\begin{example}\label{rho Sn}
   Adding to Example 6.2.7 in \cite{Bhagwat_2025}. Consider the case in Example \ref{unique F example} (2). Let $L_k=K(\alpha_1,\dots, \alpha_k)$ and let $L=L_1$ and $L_0=K$. So Galois closure of $L/K$ is $\tilde{L}=L_{n-1}$.  \smallskip

   Now $L_{k+1}=L_{k}(\alpha_{k+1})$. One can verify that the minimal polynomial of $\alpha_{k+1}$ over $L_{k}$ has degree $n-k$ and has the roots $\alpha_{k+1},\alpha_{k+2},\dots, \alpha_n$. Also $\alpha_i\not \in L_{k+1}$ for $i>k+1$ and $k\leq n-3$. Thus for $0\leq k\leq n-2$, Galois closure of $L_{k+1}/L_{k}$ is $L_{n-1}/L_{k}$ with Galois group $\mathfrak{S}_{n-k}$. \smallskip
   
   We have for $0\leq k\leq n-3$ that $[L_{k+j}:L_{k}]=\ ^{n-k}P_j$ and $\rho_{L_{k}}(L_{k+j},L_{k+1})=j$ where\\
   $1\leq j\leq n-2-k$. Also by Theorem 3 in \cite{krithika2023root}, we have $r_{L_{k}}(L_{k+j})=j!$. Thus by proof of Theorem 3.2.4 in \cite{Bhagwat_2025}, we have $\rho_{L_k}(L_{k+j},L_{k+l})=\ ^{j}C_{l}\ r_{L_k}(L_{k+l})=\ ^{j}C_{l}\ l!=\ ^{j}P_{l}$ where $1\leq l\leq j$.\smallskip

   We have $r_{L_k}(L_{k+j})=j!= l!\ (j-l)!\ ^jC_l=r_{L_k}(L_{k+l})r_{L_{k+l}}(L_{k+j})\ ^jC_l$. Also for $1\leq m< l$, $\rho_{L_k}(L_{k+j},L_{k+l})=\ ^jP_m\ ^{j-m}P_{l-m}=\rho_{L_k}(L_{k+j},L_{k+m})\rho_{L_{k+m}}(L_{k+j},L_{k+l})$.\end{example}

   \begin{remark}
  Consider $L/K$ and $M/M'/K$. Then $\rho_K(M',L)\ |\ \rho_K(M,L)$ is not true in general. For example consider the case in Example \ref{rho Sn}. Now $\rho_K(L_j,L_1)=j$ and $\rho_K(L_{j+1},L_1)=j+1$. Now $j\nmid (j+1)$ for $j\geq 2$.
\end{remark}

\begin{remark}
    Consider $M/L/L'/K$. Then the statements $r_K(L')r_{L'}(L)|r_K(L)$ and $\rho_K(M,L)=\rho_K(M,L') \rho_{L'}(M,L)$ are not true in general. Let $M=\mathbb{Q}(\sqrt[8]{2}),L=\mathbb{Q}(\sqrt[4]{2}), L'=\mathbb{Q}(\sqrt{2}), K=\mathbb{Q}$. Then $\rho_K(M,L)=r_K(L)=2\neq 4=2\cdot 2=r_K(L')r_{L'}(L)=\rho_K(M,L') \rho_{L'}(M,L)$.
\end{remark}

We can generalize Cluster Magnification result, Proposition \ref{cluster magnification} for root capacity by using results from \cite{Bhagwat_2025}. 

\begin{proposition}
    
\label{root capacity mag}

(Root Capacity Magnification) Consider $M/L/K$. Suppose $M'/K$ and $L'/K$ are obtained by strong cluster magnification from $M/K$ and $L/K$ respectively through the same $F/K$ with magnification factor $d$. Then $\rho_K(M',L')=d\ \rho_K(M,L)$. 
    
\end{proposition}

\begin{proof}
  Let $\rho_K(M,L)=a\ r_K(L)$ and $\rho_{K}(M',L')=a' \ r_{K}(L')$ and $\rho_{F}(M',L')=a''\ r_F(L')$ where $a, a', a''$ are as in Definition \ref{root capacity}. From Proposition \ref{cluster magnification}, we have $r_K(L')=d\ r_{K}(L)$.\smallskip

By Definition \ref{SCM}, $M'=MF$ and $L'=LF$ where both the pairs $\tilde{M}$ and $F$ and $\tilde{L}$ and $F$ are linearly disjoint over $K$. By Lemma 8.1.7 in \cite{Bhagwat_2025}, $r_F(LF)=r_K(L)$. By Base Change Theorem for root capacity, Theorem 8.2.2 in \cite{Bhagwat_2025}, $\rho_{F}(MF,LF)=\rho_K(M,L)$. Hence $a=a''$.\smallskip

Now by Corollary 8.1.5 in \cite{Bhagwat_2025}, $P/K$ is isomorphic to $LF/K$ $\iff$ $P/F$ is isomorphic to $LF/F$. Since $a'$ is number of distinct fields inside $MF$ isomorphic to $LF$ over $K$ and $a''$ is number of distinct fields inside $MF$ isomorphic to $LF$ over $F$. Hence $a'=a''$. Therefore $a=a'$ and we are done. \end{proof}

\smallskip

In \cite{krithika2023root}, the notion of cluster tower of polynomials is introduced. See Section 5.2 in \cite{Bhagwat_2025} for the group theoretic formulation. By using Section 3.2 in \cite{Bhagwat_2025}, we give the following field theoretic formulation.\smallskip

 \textbf{Cluster tower of an extension:} Consider $L/K$. Consider an ordering $(L_1,L_2,\dots,L_s)$ of distinct fields isomorphic to $L$ over $K$ where $s=s_K(L)$. Now consider the following cluster tower of fields
terminating at the Galois closure $\tilde{L}$.
\[ K \subseteq L_1 \subseteq L_1 L_2 \subseteq \dots \subseteq L_1 L_2 \cdots L_s = \tilde{L}.\]

We define the length of tower as defined in Section 2.2 in \cite{jaiswal2025minimal}. The length of tower is the number of strict inclusions in the tower. We also have the notion of degree sequence of the tower which is the ordered sequence of the degrees of the distinct extensions over $K$ in the tower. Clearly length of tower $\leq s$. \smallskip

Example 5.1.3 in \cite{Bhagwat_2025} demonstrates that both the degree sequence and length of tower are dependent on the ordering of the $L_i$’s.

\begin{proposition}\label{tower capacity}
 Suppose there exists a permutation $(i_1,i_2,\dots, i_s)$ of $(1,2,\dots, s)$ such that \[ K \subseteq L_{i_1} \subseteq L_{i_1} L_{i_2} \subseteq \dots \subseteq L_{i_1} L_{i_2} \cdots L_{i_s} = \tilde{L}.\] is a cluster tower for $L/K$ of length $s$. Then for each $0\leq a\leq s$ there exists an $M/K$ such that $\rho_K(M,L)=a\ r_K(L)$.

\end{proposition}

\begin{proof}
   For $a=0$, $M=K$ works. We claim that for $a\geq 1$, $M= L_{i_1} L_{i_2}\cdots L_{i_a}$ works. Since length of tower is $s$, we have each field in the tower to be a proper subset of successive field. Hence $\rho_{K}(L_{i_1} L_{i_2}\cdots L_{i_a}, L)\geq a\ r_K(L)$. If $\rho_{K}(L_{i_1} L_{i_2}\cdots L_{i_a}, L)> a\ r_K(L)$, then because of proper containment at each step, we will have $\rho_{K}(L_{i_1} L_{i_2}\cdots L_{i_s}, L)> s\ r_K(L)=[L:K]$ which is a contradiction. Thus $\rho_K(M,L)= a\ r_K(L)$.\end{proof}\smallskip

   \begin{proposition}
       \label{SCM cluster tower}
  
            Let $M/K$ be obtained by strong cluster magnification from $L/K$ through $F/K$. Then \[ K \subseteq L_1 \subseteq L_1 L_2 \subseteq \dots \subseteq L_1 L_2 \cdots L_s\] is a cluster tower for $L/K$ of length $l$ if and only if \[ K \subseteq L_1F \subseteq L_1 L_2F \subseteq \dots \subseteq L_1 L_2 \cdots L_sF\] is a cluster tower for $M/K$ of length $l$. If degree sequence of first tower is $(a_0,a_1,\dots, a_{l-1})$, then degree sequence of second tower is $(a_0 d,a_1 d,\dots, a_{l-1} d)$ where $d=[F:K]$.

 \end{proposition}

   \begin{proof}
By Corollary 8.1.5 in \cite{Bhagwat_2025}, the distinct fields inside $\bar{K}$ isomorphic to $M$ over $K$ are precisely $L_iF$ for $1\leq i\leq s$. Thus it is enough to show that, for any $1\leq k\leq s-1$, $L_1 L_2 \cdots L_k = L_1 L_2 \cdots L_k L_{k+1}$ if and only if $L_1 L_2 \cdots L_k F = L_1 L_2 \cdots L_k L_{k+1} F$.\smallskip

 Since $\tilde{L}\cap F=K$. Hence for $1\leq k\leq s$, we have $L_1 L_2 \cdots L_k \cap F=K$ and $L_1 L_2 \cdots L_k F \cap \tilde{L} = L_1 L_2 \cdots L_k$. If for any $1\leq k\leq s-1$, $L_1 L_2 \cdots L_k F = L_1 L_2 \cdots L_k L_{k+1} F$. Then $L_1 L_2 \cdots L_k = L_1 L_2 \cdots L_k F \cap \tilde{L} = L_1 L_2 \cdots L_k L_{k+1} F \cap \tilde{L}= L_1 L_2 \cdots L_k L_{k+1}$. We also have $[L_1 L_2 \cdots L_k F : K]=[L_1 L_2 \cdots L_k : K][F:K]$.\end{proof}

  \subsection{Intersection Indicium : A Generalization of Ascending Index}\label{intersection indicium}

  The following is Theorem 9.0.4 in \cite{Bhagwat_2025} referred to as the Ascending Index Magnification Theorem.

  \begin{proposition}
      
  \label{asc ind mag thm}
       Let $M/K$ be obtained by strong cluster magnification from $L/K$ with magnification factor $d$. Then $[M:K]=d\ [L:K]$ and $t_K(M)=d \ t_K(L)$.
  \end{proposition}

We define the following notion.

\begin{definition}
    Consider extensions $L/K$ and $M/K$. Let $L_1,L_2,\dots , L_a$ be all the distinct fields inside $M$ isomorphic to $L$ over $K$. Let $F'=\cap_{i=1}^a L_i$. We define intersection indicium of $M$ with respect to $L$ (with base field $K$ fixed) as $\tau_K(M,L)=[F':K]$. If none of the fields isomorphic to $L/K$ is contained in $M$, then we define $\tau_K(M,L)=0$.
\end{definition}

\begin{remark}\label{tau group theoretic}
    Equivalently $\tau_K(M,L)=[G:H']$ where $G=Gal(\tilde{L}/K)$ (where $\tilde{L}$ is Galois closure of $L/K$) and $H'$ is subgroup of $G$ generated by subgroups $H_1, H_2,\dots, H_a$ which fix $L_1,L_2,\dots,L_a$ respectively.  
\end{remark}

\begin{remark}\label{rem tau}
    Suppose $M/K$ and $M'/K$ are isomorphic over $K$ and $L/K$ and $L'/K$ are isomorphic over $K$. Then $\tau_K(M,L)=\tau_K(M',L')$ (If the isomorphism from $M$ to $M'$ maps $L_1,L_2,\dots,L_a$ to $L_{i_1},L_{i_2},\dots, L_{i_a}$ respectively, then it maps $\cap_{j=1}^a L_j$ to $\cap_{j=1}^a L_{i_j}$).
\end{remark}

The following are certain basic properties which demonstrate that intersection indicium is indeed a generalization of ascending index.

\begin{proposition}\label{tau prop}
Consider an extension $L/K$.

    \begin{enumerate}
        \item $\tau_K(L,L)=[L:K]$. Also $\tau_K(L,K)=1$. If $L\neq K$ then $\tau_K(K,L)=0$.\smallskip

        \item $\tau_K(\tilde{L},L)=t_K(L)$ where $\tilde{L}$ is Galois closure of $L/K$.  \smallskip

        \item For any extension $M/K$, $t_K(L)\ |\ \tau_K(M,L)$.\smallskip

        \item For any $M/K$, if $\tau_K(M,L)\neq 0$ then $\tau_K(M,L)\ |\ [L:K]$. Thus if $\tau_K(M,L)\neq 0$ then we have $\tau_K(M,L)=b\ t_K(L)$ where $b\ |\ u_K(L)$.
\end{enumerate}\end{proposition}

\begin{proof}\hfill

\begin{enumerate}
    \item Now $M=L$ contains only one field isomorphic to $L/K$ that is $L$. Thus $F'=L$. Now if $L\neq K$ then $M=K$ does not contain any field isomorphic to $L/K$.\smallskip

    \item We have that $M=\tilde{L}$ contains all the $s_K(L)$ many distinct fields isomorphic to $L/K$. Thus $F'=F$ by Proposition \ref{unique F} (1) where $F/K$ is the unique intermediate extension of $L/K$ which is Galois with maximum possible degree. Since $t_K(L)=[F:K]$, we are done. \smallskip

    \item By Proposition \ref{unique F} (1) we have that $F$ is the intersection of all the $s_K(L)$ many distinct fields isomorphic to $L/K$. Thus $F\subset F'$. Hence $[F:K]\ |\ [F':K]$.\smallskip

    \item Since $\tau_K(M,L)=[\cap_{i=1}^a L_i : K]$. Thus $\tau_K(M,L)$ divides $[L_1:K]=[L:K]$. Since $[L:K]=t_K(L)u_K(L)$ we are done.\end{enumerate} \end{proof}

We can generalize Ascending Index Magnification result, Proposition \ref{asc ind mag thm} for intersection indicium.

\begin{proposition}
    
(Intersection Indicium Magnification) \label{tau magn} Consider $M/L/K$. Suppose $M'/K$ and $L'/K$ are obtained by strong cluster magnification from $M/K$ and $L/K$ respectively through the same $F/K$ with magnification factor $d$. Then $\tau_K(M',L')=d\ \tau_K(M,L)$. 
\end{proposition}

\begin{proof}
    By proof of Proposition \ref{root capacity mag}, the distinct fields inside $M'$ isomorphic to $L'/K$ are precisely $\{L_iF\}_{i=1}^a$ where $\{L_i\}_{i=1}^a$ are fields inside $M$ isomorphic to $L/K$. Let $G'=Gal(\tilde{L'}/K)$ and $G=Gal(\tilde{L}/K)$ and $R=Gal(F/K)$. By Proposition 4.1.5 (3) in \cite{Bhagwat_2025}, we can identify $G'$ with $G\times R$. So $Gal(\tilde{L'}/\tilde{L})=1\times R$. If $H_i= Gal(\tilde{L}/L_i)$ then $H_i\times R=Gal(\tilde{L'}/L_i)$. Also $G\times 1=Gal(\tilde{L'}/F)$. Thus $(H_i\times R)\cap (G\times 1)=H_i\times 1=Gal(\tilde{L'}/(L_iF))$. Hence $Gal(\tilde{L'}/(\cap_{i=1}^a L_iF))$ is generated by subgroups $\{H_i\times 1\}_{i=1}^a$ which is $H'\times 1$ where $H'$ is subgroup of $G$ generated by subgroups $\{H_i\}_{i=1}^a$. Also $Gal(\tilde{L'}/(\cap_{i=1}^a L_i))=H'\times R$. Since $(H'\times R)\cap (G\times 1)=H'\times 1$. Hence $(\cap_{i=1}^a L_i)F=(\cap_{i=1}^a L_iF)$. Since $\tilde{L}\cap F=K$ we have $(\cap_{i=1}^a L_i)\cap F=K$ where $F/K$ is Galois. Thus we have $\tau_K(M',L')=[(\cap_{i=1}^a L_iF):K]=[(\cap_{i=1}^a L_i)F:K]=[(\cap_{i=1}^a L_i):K][F:K]=\tau_K(M,L)\ d$.
\end{proof}



\subsection{Compositum Indicium : A Generalization of Degree of Extension}

\label{compositum indicium}

Just like we defined intersection indicium, we define the following notion.

\begin{definition}
    Consider extensions $L/K$ and $M/K$. Let $L_1,L_2,\dots , L_a$ be all the distinct fields inside $M$ isomorphic to $L$ over $K$. Let $L_M=\Pi_{i=1}^a L_i$. We define compositum indicium of $M$ with respect to $L$ (with base field $K$ fixed) as $\gamma_K(M,L)=[L_M:K]$. If none of the fields isomorphic to $L/K$ is contained in $M$, then we define $\gamma_K(M,L)=0$.\smallskip

    Note that $L_M\subset M\cap \tilde{L}$.
\end{definition}

\begin{remark}\label{gamma group theoretic}
    Equivalently $\gamma_K(M,L)=[G:T]$ where $G=Gal(\tilde{L}/K)$ (where $\tilde{L}$ is Galois closure of $L/K$) and $T$ is subgroup of $G$ which is intersection of subgroups $H_1, H_2,\dots, H_a$ which fix $L_1,L_2,\dots,L_a$ respectively.  
\end{remark}

\begin{remark}\label{rem gamma}
    Suppose $M/K$ and $M'/K$ are isomorphic over $K$ and $L/K$ and $L'/K$ are isomorphic over $K$. Then $\gamma_K(M,L)=\gamma_K(M',L')$ (If the isomorphism from $M$ to $M'$ maps $L_1,L_2,\dots,L_a$ to $L_{i_1},L_{i_2},\dots, L_{i_a}$ respectively, then it maps $\Pi_{j=1}^a L_j$ to $\Pi_{j=1}^a L_{i_j}$).
\end{remark}

The following are certain basic properties which demonstrate that compositum indicium is indeed a generalization of degree of extension.

\begin{proposition}\label{gamma prop}
Consider an extension $L/K$.

    \begin{enumerate}
        \item $\gamma_K(L,L)=[L:K]$ and $\gamma_K(\tilde{L},L)=[\tilde{L}:K]$ where $\tilde{L}$ is Galois closure of $L/K$. Also $\gamma_K(L,K)=1$. If $L\neq K$ then $\gamma_K(K,L)=0$.\smallskip

        \item For any extension $M/K$, $[L:K]\ |\ \gamma_K(M,L)$. If $\gamma_K(M,L)\neq 0$ then $\gamma_K(M,L)\ |\ [\tilde{L}:K]$ and $\gamma_K(M,L)\ |\ [M:K]$. 
\end{enumerate}\end{proposition}



    


\begin{proposition}
    
(Compositum Indicium Magnification) \label{gamma magn} Consider $M/L/K$. Suppose $M'/K$ and $L'/K$ are obtained by strong cluster magnification from $M/K$ and $L/K$ respectively through the same $F/K$ with magnification factor $d$. Then $\gamma_K(M',L')=d\ \gamma_K(M,L)$. 
\end{proposition}

\begin{proof}
    By proof of Proposition \ref{root capacity mag}, the distinct fields inside $M'$ isomorphic to $L'/K$ are precisely $\{L_iF\}_{i=1}^a$ where $\{L_i\}_{i=1}^a$ are fields inside $M$ isomorphic to $L/K$. Hence $L_MF=(\Pi_{i=1}^a L_i)F=(\Pi_{i=1}^a L_iF)=L'_{M'}$. Since $\tilde{L}\cap F=K$ we have $L_M\cap F=K$ where $F/K$ is Galois. Thus we have $\gamma_K(M',L')=[L'_{M'}:K]=[L_MF:K]=[L_M:K]\cdot [F:K]=\gamma_K(M,L)\ d$.
\end{proof}



\subsection{Some observations and an important example}

The following connects the concepts of root capacity and intersection indicium and compositum indicium.
    \begin{proposition}\label{tau prop 2} Consider extensions $L/K$ and $M/K$.

    \begin{enumerate}

\item $\rho_K(M,L)=\rho_K(L_M,L),\ \tau_K(M,L)=\tau_K(L_M,L),\ \gamma_K(M,L)=\gamma_K(L_M,L)$.

\item $\rho_K(M,L)=r_K(L)\iff \tau_K(M,L)=[L:K]\iff  \gamma_K(M,L)=[L:K]$.
    
\item $\rho_K(M,L)=[L:K]\implies\tau_K(M,L)=t_K(L)$. The converse is false as demonstrated in Example \ref{tau examples}.

\item  $\rho_K(M,L)=[L:K]\iff \gamma_K(M,L)=[\tilde{L}:K]$. 

\item $\tau_K(M,L) \mid \gamma_K(M,L)$ if $\tau_K(M,L)\neq 0$. The statement $\rho_K(M,L) \mid \gamma_K(M,L)$ is false in general as demonstrated later in Remark \ref{rho not divide}.
        
    \end{enumerate}
        
    \end{proposition}



    

\begin{example}\label{tau examples}
     Consider $L/K$ which has Galois closure $\tilde{L}$ with Galois group as $\mathfrak{S}_n$ for $n>2$. By the argument in Example \ref{unique F example} (2), we have that intersection of any two distinct fields isomorphic to $L/K$ is $K$. Thus for an extension $M/K$, if $\tau_K(M,L)\neq 0$ or $[L:K]$ then $\tau_K(M,L)=1=t_K(L)$.\smallskip

     We also have that if $M$ is compositum of any two distinct fields isomorphic to $L/K$ then $\rho_K(M,L)=2=2\ r_K(L)$ by Example \ref{rho Sn}. Thus $\rho_K(M,L)\neq n=[L:K]$.

\end{example}

Some other properties which are easy to verify.

\begin{proposition} \label{tau prop 3} Consider an extension $L/K$.
    \begin{enumerate}
        \item Consider $M/M'/K$. Then $\tau_K(M,L)\ |\ \tau_K(M',L)$ if $\tau_K(M,L)\neq 0$. Also $\gamma_K(M',L)\ |\ \gamma_K(M,L)$ if $\gamma_K(M',L)\neq 0$.\smallskip

        \item Consider $M/K$ and $K'/K$ such that $K'\subset (L\cap M)$, Then $\tau_K(M,L)\ |\ (\tau_{K'}(M,L)\cdot [K':K])$ if $\tau_K(M,L)\neq 0$. Also $(\gamma_{K'}(M,L)\cdot [K':K])\ |\ \gamma_{K}(M,L)$ if $\gamma_{K'}(M,L)\neq 0$.

    \end{enumerate}
\end{proposition}

 The following result is an important example.

\begin{proposition}\label{tau interesting eg}

    Consider the case in Example \ref{unique F example} (1) for $n$ odd. Let $L=\mathbb{Q}(a)$ and $M=\mathbb{Q}(a, \zeta^l)$ for an $l|n$. Then $\rho_{\mathbb{Q}}(M,L)=n/l$ and $\tau_{\mathbb{Q}} (M,L)=l$ and thus $\tau_{\mathbb{Q}} (M,L)\rho_{\mathbb{Q}} (M,L)=[L:\mathbb{Q}]$. Also $\gamma_{\mathbb{Q}}(M,L)=n\phi(n/l)$ where $\phi$ is the Euler totient function.
\end{proposition}

\begin{proof}

    

Suppose $\mathbb{Q}(a\zeta^i)\subset M$. Then $\zeta^i\in \mathbb{Q}(a,\zeta^l)$ and so $\zeta^i\in \mathbb{Q}(a,\zeta^l)\cap \mathbb{Q}(\zeta)$. Since $\mathbb{Q}(a)\cap \mathbb{Q}(\zeta)=\mathbb{Q}$. Thus $\mathbb{Q}(a,\zeta^l)\cap \mathbb{Q}(\zeta)=\mathbb{Q}(\zeta^l)$ and so $\zeta^i\in \mathbb{Q}(\zeta^l)$. Let $d=gcd(i,l)$. Then $\zeta^d\in \mathbb{Q}(\zeta^l)$. Now we have $[\mathbb{Q}(\zeta^l):\mathbb{Q}]=\phi(n/l)$. Since $d|l$, we have $d|n$. Thus $[\mathbb{Q}(\zeta^d):\mathbb{Q}]=\phi(n/d)$. Since $\mathbb{Q}(\zeta^d)\subset \mathbb{Q}(\zeta^l)$. Thus $\phi(n/d)\ |\ \phi(n/l)$. Since $d|l$ we have $(n/l) | (n/d)$. Thus $\phi(n/l)\ |\ \phi(n/d)$ which implies that $\phi(n/d)= \phi(n/l)$. Since $\phi(n/d)= \phi((n/d)/(l/d))$ and $n$ is odd. Thus $n/d=n/l$ that is $d=l$. Thus we have $l|i$ (Refer Proposition 5.1.2 in \cite{Bhagwat_2025} for a list of properties of $\phi$ that we have used here). Thus the distinct fields inside $M$ isomorphic to $L/\mathbb{Q}$ are precisely $\{\mathbb{Q}( a\zeta^{yl})\}_{y=0}^{(n/l -1)}$. Since $n$ is odd, we have $r_{\mathbb{Q}}(L)=1$. Thus $\rho_{\mathbb{Q}} (M,L)=n/l$.

\smallskip

Now let $F'$ be intersection of all these distinct fields inside $M$ isomorphic to $L/\mathbb{Q}$. Then we have $\tau_{\mathbb{Q}}(M,L)=[F':K]$. Clearly $a^{n/l}\in F'$. We will show that $F'=\mathbb{Q}(a^{n/l})$. Suppose $\gamma\in F'$. So $\gamma\in \mathbb{Q}(a)\cap \mathbb{Q}( a\zeta^{l})$. Hence $\gamma=a_0 +a_1 a + \dots + a_{n-1} a^{n-1}=b_0 +b_1 (a\zeta^{l}) + \dots + b_{n-1} (a\zeta^{l})^{n-1}$ for $a_j,b_j\in \mathbb{Q}$ for all $0\leq j\leq n-1$. By arguments as in Example \ref{unique F example} (1), we have for all $0\leq j\leq n-1$ that $a_j=b_j (\zeta^{l})^j$. Since $n$ is odd, $a_j\neq 0$ implies $n\ |\ (lj)$ that is $(n/l)\ |\ j$. Thus if $a_j\neq 0$ then $a^j\in \mathbb{Q}(a^{n/l})$ and thus $\gamma\in \mathbb{Q}(a^{n/l})$. Since $l|n$, we have $\tau_{\mathbb{Q}}(M,L)=[\mathbb{Q}(a^{n/l}):K]=l$. Now $L_M=\Pi_{y=0}^{(n/l-1)}\ \mathbb{Q}(a\zeta^{yl})$. Thus $L_M=\mathbb{Q}(a, \zeta^l)=M$. Thus $\gamma_{\mathbb{Q}}(M,L)=[L_M:\mathbb{Q}]=[\mathbb{Q}(a, \zeta^l):\mathbb{Q}]$. Since $\mathbb{Q}(a)\cap \mathbb{Q}(\zeta^l)=\mathbb{Q}$ and $\mathbb{Q}(\zeta^l)$ is Galois. Thus $\gamma_\mathbb{Q}(M,L)=[\mathbb{Q}(a):\mathbb{Q}]\cdot [\mathbb{Q}(\zeta^l):\mathbb{Q}]=n\phi(n/l)$.\end{proof}

\begin{remark}
    Proposition \ref{tau interesting eg} is a generalization of Proposition 2.4 (3) in \cite{jaiswal2025minimal}.
\end{remark}

\begin{remark}
    For extensions $L/K$ and $M/K$, the statement $\tau_K(M,L)\rho_K (M,L)=[L:K]$ is not true in general which is demonstrated by the case in Example \ref{tau examples}.
\end{remark}

\subsection{Inverse Root Capacity Problem}\label{Inverse Root Capacity Problem}
We begin by stating the following lemma.

\begin{lemma}[Lemma 3.1.4, \cite{Bhagwat_2025}]\label{3.1.4} 

Let $K$ be a perfect field and $K'/K$ be a finite extension. If for a group $G$, direct products $G^n$ are realisable as Galois group over $K$ for each $n\in \mathbb{N}$ then they are realisable as Galois groups over the field $K'$.
    
\end{lemma}

The following lemma follows from the proof in \cite{Bhagwat_2025} of Lemma \ref{3.1.4}  (essentially the proof of Theorem 4.2 in \cite{conrad2023galois}). We provide the proof for the sake of completeness.

\begin{lemma}\label{arbitrary lemma}
     Suppose for a group $G$, direct product $G^m$ is realizable as a Galois group over $K$ for an $m\in \mathbb{N}$. Then we get $m$-many Galois extensions of $K$ inside $\bar{K}$ which are pairwise non-isomorphic over $K$  and are pairwise linearly disjoint over $K$ with each having Galois group $G$ over $K$. \end{lemma}

\begin{proof}
    Now $G^m$ is realizable over $K$, say for $E/K$ Galois, we have ${\rm Gal}(E/K)\cong G^m$. We have normal
subgroups $N_i = G \times G \times \dots \times {1} \times \dots \times G$ of $G^m$ for $1 \leq i \leq m$ where the $i$th coordinate is trivial and there is no restriction in other
coordinates. So $N_i \cong G^{m-1}$. Let $E_i$ be the subfield of $E$ corresponding to
$N_i$, so $E_i/K$ is Galois with ${\rm Gal}(E_i/K) \cong G^m / N_i \cong G$.\smallskip

We observe that $N_i$ are not conjugate to each other in $G^m$ and they pairwise generate $G^m$. Hence $E_i$ are not isomorphic to each other over $K$ and are pairwise linearly disjoint over $K$.
 \end{proof}




\begin{proof}[Proof of Theorem \ref{inverse cluster size improv}]
We proceed in the same way as the author and Bhagwat proceeded in the proof of Theorem \ref{n,r} in \cite{Bhagwat_2025}. Suppose $r=1$. By Lemma \ref{Sn}, there exist infinitely many pairwise linearly disjoint extensions over $K$ of degree $n$ with Galois closures having Galois group ${\mathfrak S}_{n}$ over $K$. Hence we obtain infinitely many extensions of degree $n$ which are pairwise non-isomorphic and pairwise linearly disjoint over $K$ and these extensions have cluster size 1.\smallskip

Suppose $r>1$. We have that there exists $L/K$ with $G=Gal(\tilde{L}/K)$ solvable and $[L:K]=n$ and $r_K(L)=r$. Consider subgroup $H=Gal(\tilde{L}/L)\subset G$. Since direct product of solvable groups is solvable, direct products $G^m$ for $m\in \mathbb{N}$ are solvable. By Shafarevich's theorem (\cite{Shafarevich1989}), $G^m$ for $m\in \mathbb{N}$ are realizable as Galois groups over $\mathbb{Q}$. Hence by Lemma \ref{3.1.4}, $G^m$ for $m\in \mathbb{N}$ are realizable as Galois groups over number field $K$.\smallskip

By Lemma \ref{arbitrary lemma}, for any $m\in \mathbb{N}$, we get $E_i$ for $1\leq i\leq m$ which are pairwise non-isomorphic over $K$  and are pairwise linearly disjoint over $K$ with each having Galois group $G$ over $K$. Thus for $1\leq i\leq m$, we have $E_i^H$ (which correspond to subgroups $G \times G \times \dots \times H \times \dots \times G$ of $G^m$ where the $i$th coordinate is an element of $H$ and there is no restriction in other
coordinates) which are pairwise non-isomorphic over $K$  and are pairwise linearly disjoint over $K$ with each having Galois closures with Galois group $G$ over $K$. Each $E_i^H/K$ has degree $n$ and cluster size $r$.\end{proof}

\begin{remark}
    In the above proof, for any given $m\in\mathbb{N}$, we get $m$ many extensions over $K$, with degree $n$ and cluster size $r$, which are pairwise non-isomorphic over $K$. Thus for any given $m$, we get $mn/r$ many distinct extensions over $K$, with degree $n$ and cluster size $r$.
\end{remark}

\begin{remark}
The cases $r=1$ for $n$ odd and $r=2$ for $n$ even have other obvious isomorphism classes too distinct from the ones obtained in above theorem. Namely the ones considered in Example \ref{unique F example} (1) and Theorem \ref{inv min gen xn} (1) where Galois group of Galois closure of extension is $\mathbb{Z}/n \mathbb{Z} \rtimes (\mathbb{Z}/n \mathbb{Z})^{\times}$. Observe that even this group is solvable, so a similar story unfolds.
\end{remark}



\begin{remark}\label{remark inv root cap}
   Assuming the condition $\rho\neq n-1$ is necessary in Theorem \ref{inv root cap}. Suppose $L/K$ is degree $n$ extension. Then there doesn't exist $M/K$ such that $\rho_K (M,L)=n-1$. Assume on the contrary that such $M/K$ exists. Since $r_K(L)|\rho_K (M,L)$ and $r_K(L)| n$. Thus $r_K(L)=1$ and $s_K(L)=n$. Let $L=K(\alpha)$ and $f$ be minimal polynomial for $\alpha$ over $K$. Since $\rho_K (M,L)=n-1$, $M$ contains $n-1$ roots of $f$. Since sum of all roots of $f$ lies in $K$. Hence $M$ contains the $n$-th root as well which is a contradiction.
\end{remark}
\smallskip

By results in \cite{volklein1996groups} on hilbertian fields and by the final proposition in \cite{perlisroots}, we have the following.

\begin{lemma}\label{Sn}
Let $K$ be a number field and $n\in \mathbb{N}$. Then we have ${\mathfrak S}_{n}$ to be realizable as a Galois group for infinitely many pairwise linearly disjoint Galois extensions over $K$. Thus there exist infinitely many pairwise linearly disjoint extensions over $K$ of degree $n$ with Galois closures having Galois group ${\mathfrak S}_{n}$ over $K$.
\end{lemma}

\begin{lemma}[Lemma 3, \cite{krithika2023root}]\label{lemma 3 in vanchi}

Let $K$ be a number field and $L/K$ be a finite extension. For any positive integer
$d \geq 2$, there exist infinitely many cyclic Galois extensions $L'/K$ of degree $d$ such that $L$ and $L'$ are linearly
disjoint over $K$.
    
\end{lemma}

\begin{proof}[Proof of Theorem \ref{inv root cap}]

We discuss two approaches to prove the result.\smallskip

The First Approach: This uses Proposition \ref{root capacity mag} but excludes the cases (1) $n=2r$ and (2) $\rho=n-r$ for $n>2r$.\smallskip

  Let $n/r=s$ and $\rho/r = a$ and let $s\neq2$ and $a\neq s-1$. By Lemma \ref{Sn}, there exists an irreducible polynomial $f$ over $K$ of degree $s$ with Galois group ${\mathfrak S}_s$. This $f$ satisfies $r_K(f)=1$. Let roots of $f$ be $\alpha_i$ for $1\leq i\leq s$. For $1 \leq k \leq s-1$, let $L_k=K(\alpha_1,\dots, \alpha_k)$. As noted in Example 6.2.7 in \cite{Bhagwat_2025}, $r_K(L_1)=1$ and $[L_1:K]=s$ and $L_{s-1}$ is Galois closure of $L_1/K$ and $\rho_K(L_k,L_1)=k$ for $1\leq k\leq (s-2)$ and $\rho_K(L_{s-1},L_1)=s$.\smallskip

By Lemma \ref{lemma 3 in vanchi}, there exists $F/K$ Galois of degree $r$ such that $L_{s-1}$ and $F$ are linearly disjoint over $K$. Let $L=L_1 F$. By Proposition \ref{cluster magnification}, $[L:K]=rs=n$ and $r_K(L)=r$. Then for $a=0$, $M=K$ works. Let $M=L_a F$ for $a\neq 0, s$ and let $M=L_{s-1} F$ for $a=s$. By Proposition \ref{root capacity mag}, $\rho_K(M,L)=r\  \rho_K (L_a, L_1)= ar=\rho$ for $a\neq 0,s$ and $\rho_K(M,L)=r\ \rho_K(L_{s-1}, L_1)=sr=\rho$ for $a=s$. \smallskip

The Second Approach: This gives a complete answer.

\smallskip

 For $r=1$ we use the first approach itself. For $r>1$, we proceed as we did in the proof of Theorem \ref{n,r} in \cite{Bhagwat_2025}. Thus we have that there exists an extension $L/K$ with Galois closure $\tilde{L}$ such that $G=Gal(\tilde{L}/K)$ is solvable such that action of the group is transitive on $n$ points, and a point stabiliser fixes precisely $r$ points. We have $G=(\Z/r\Z)^s \rtimes \Z / s\Z$ with semidirect product group law \[ ((a_1,\dots, a_s),b)\cdot((c_1,\dots,c_s),d)=((a_1,\dots, a_s) + (b\cdot (c_1,\dots, c_s)), b+d), \]
  where $b\cdot(c_1,\dots, c_s)=(c_{b+1},\dots, c_s,c_1,\dots, c_{b})$ for $b\neq 0$ \&  $0\cdot(c_1,\dots, c_s)=(c_1,\dots, c_s)$. \smallskip
  
  One can verify that $((a_1,\dots,a_s),b)^{-1}=((-a_{s-b+1},\dots,-a_{s}, -a_{1},\dots, -a_{s-b}),-b)$ and \[((a_1,\dots, a_s),b)\cdot((c_1,\dots,c_{s-1},0),0)\cdot ((a_1,\dots,a_s),b)^{-1}= ((c_{b+1},\dots, c_{s-1}, 0, c_1,\dots,c_b),0).\]

 Any point stabiliser is isomorphic to $(\Z/r\Z)^{s-1}$. We have $[L:K]=n$ and $r_K(L)=r$. The $s=s_K(L)$ many subgroups of $G$ fixing the $s$ many distinct fields $L_i$'s isomorphic to $L/K$ are $H_i=((\Z/r\Z)^{i-1}\times 0 \times (\Z/r\Z)^{s-i} )\times 0$ for $1\leq i\leq s$. Observe that \[ G \supsetneq H_1 \supsetneq H_1 \cap H_2 \supsetneq \dots \supsetneq H_1 \cap H_2 \cap \cdots \cap H_s = 0.\] 

 Thus by the group theoretic formulation for cluster towers in Section 5.2 in \cite{Bhagwat_2025}, we have \[ K \subseteq L_1 \subseteq L_1 L_2 \subseteq \dots \subseteq L_1 L_2 \cdots L_s = \tilde{L}.\] is a cluster tower for $L/K$ of length $s$. Hence we are done by Proposition \ref{tower capacity}.
\end{proof}

\begin{remark}
    Consider the case in the first approach above. Now $L/K$ is obtained by strong cluster magnification from $K(\alpha_1)/K$ through $F/K$. Thus from the alternate proof of cluster magnification theorem in Section 8 in \cite{Bhagwat_2025}, we have $s_K(L)=s_K(K(\alpha_1))=s$ and $K(\alpha_i) F$ for $1\leq i\leq s$ are the $s$ many distinct fields isomorphic to $L$ over $K$. Hence we have by Proposition \ref{SCM cluster tower} that \[ K \subseteq L_1 F \subseteq L_2 F \subseteq \dots \subseteq L_{s-1} F = \tilde{L}.\] is a cluster tower for $L/K$ of length $s-1$ where $L_k=K(\alpha_1,\dots, \alpha_k)$ for $1 \leq k \leq s-1$. The degree sequence of the cluster tower is $(n, n\ ^{(s-1)} P_1,n\ ^{(s-1)} P_2,\dots, n\ ^{(s-1)} P_{(s-2)})$. The degree sequence and length of tower are independent of the ordering of the $K(\alpha_i) F$’s.
\end{remark}

\begin{remark} \label{rem for min gen set}
    In the second approach above we can compute the degree sequence of the considered cluster tower for $L/K$ which turns out to be $(n,nr, nr^2, \dots , nr^{s-1})$. It is easy to see that degree sequence and length of tower are independent of the ordering of the $L_i$’s.
\end{remark}

\begin{remark}
\label{conv countereg}

  For $1<r\leq n/3$ and $r|n$, consider the field $L/K$ in second approach above. Thus $3\leq s=n/r<n$. For $1\leq a\leq s-1$, let $M_a=L_1 L_2 \cdots L_a$. Now Galois closure of $M_a/K$ is $\tilde{L}$. Thus for $2\leq a \leq s-1$, we have that $M_a/K$ is not obtained by strong cluster magnification from $M_1/K$.\smallskip

  Now for $1\leq a\leq s-1$, the subgroup of $G$ fixing $M_a$ is $G_a= H_1\cap H_2 \cap \cdots \cap H_a$. We have the notions of unique descending chains and unique ascending chains for extensions introduced in Section 7 in \cite{Bhagwat_2025}. By using the group law we can show that for $1\leq a\leq s-1$, $N_G(G_a)=G_a^G=(\Z/r\Z)^s\times \{0\}$ where $G_a^G$ is the normal closure of $G_a$ in $G$. \smallskip

  Hence for any $1\leq a \leq s-1$ the unique descending chain for $M_a/K$ is $M_a\supsetneq \tilde{L}^{(\Z/r\Z)^s\times \{0\}} \supsetneq K$ and the unique ascending chain for $M_a/K$ is $K\subsetneq \tilde{L}^{(\Z/r\Z)^s\times \{0\}} \subsetneq M_a$ and both the unique chains coincide. Thus $r_K(M_a)=[N_G(G_a):G_a]=[G_a^G :G_a]=u_K(M_a)=r^s/r^{s-a}=r^a$ and\\
  $s_K(M_a)=[G:N_G(G_a)]=[G:G_a^G]=t_K(M_a)=s$. 
  
  \smallskip

Thus we have for $2\leq a\leq s-1$ that $M_a/K$ and the subextension $M_1/K$ serve as a counterexample for the converse of Theorem 8.4.1 in \cite{Bhagwat_2025}.
\end{remark}

\begin{remark}\label{rho not divide}
  For extensions $L/K$ and $M/K$ with $\rho_K(M,L)\neq 0$, it is not necessary that we have $\rho_K(M,L)\ |\ [M:K]$. Thus it is not even necessary that $\rho_K(M,L)\ |\  [L:K]$ and $\rho_K(M,L)\ |\  \gamma_K(M,L)$.\smallskip

  For example consider the case in Remark \ref{conv countereg}. Then $\rho_K(M_a,L)=a\ r_K(L)=ar$ and $[M_a:K]=r^a\ s$. Thus $\rho_K(M_a,L)\ |\ [M_a:K]$ if and only if $a|r^{a-1}s$ which is not always true.
\end{remark}

An improvement on Inverse root capacity problem for number fields Theorem \ref{inv root cap} (Similar improvement as in Theorem \ref{inverse cluster size improv}): Let $K$ be a number field. Given $(n,r,\rho)$ where $n>2$ and $r|n$ and $r|\rho$ and $\rho\neq n-1$. We get arbitrarily large finite families of degree $n$ extensions $L/K$ inside $\bar{K}$ with cluster size $r$, which are pairwise non-isomorphic over $K$ and are pairwise linearly disjoint over $K$, for which we have extensions $M/K$ such that $\rho_K (M,L)=\rho$.\smallskip

For $\rho\neq 0$, we get $M/K$ as an extension of $L/K$ contained in $\tilde{L}$. Thus extensions $M/K$ corresponding to extensions $L/K$ which are pairwise non-isomorphic over $K$ and pairwise linearly disjoint over $K$ are themselves pairwise non-isomorphic over $K$ and pairwise linearly disjoint over $K$.

\subsection{Inverse Intersection Indicium Problem}
\label{inv tau section}
\hfill

An improvement on Theorem \ref{t,r} (Similar improvement as in Theorem \ref{inverse cluster size improv}): Let $K$ be a number field. Let $n>2$ and $t|n$. Then we get arbitrarily large finite families of extensions $L/K$  inside $\bar{K}$ which are pairwise non-isomorphic over $K$ and are pairwise linearly disjoint over $K$ and each having degree $n$ with ascending index $t_K(L)=t$. 
\smallskip


    Example \ref{perlis tau} below demonstrates that we can't directly use the second approach used for proving Inverse root capacity problem Theorem \ref{inv root cap} to solve the Inverse intersection indicium problem.

\begin{example}\label{perlis tau}

Let $K$ be a number field and $n>2$ and $t|n$ with $t>1$. Let $r=n/t$ for $t<n$ and $r=n$ for $t=n$. By the proof of Theorem \ref{t,r} in \cite{Bhagwat_2025}, we have existence of an extension $L/K$ with $[L:K]=n$ and $t_K(L)=t$ and $r_K(L)=r$ which is the same as the field extension $L/K$ in second approach of proof of Theorem \ref{inv root cap}. Consider the notations as in that proof above.\smallskip

If $L_i$ and $L_j$ are two distinct fields isomorphic to $L/K$ then the subgroups $H_i$ and $H_j$ fixing them clearly generate the subgroup $(\Z/r\Z)^s\times \{0\}$ which is same as $H^G$ as shown in Theorem 7.3.4 in \cite{Bhagwat_2025}. Thus for an extension $M/K$, if $\tau_K(M,L)\neq 0$ or $[L:K]$ then $\tau_K(M,L)=t=t_K(L)$. In particular, for $L/K$ non-Galois ($t<n$) and $2\leq a\leq s$, we have $\tau_K(L_1 L_2\cdots L_a,L)=t=t_K(L)$.

\end{example}

Similarly one can see in the light of Example \ref{tau examples}, that the first approach of Theorem \ref{inv root cap} is also not directly useful to solve the problem.

\smallskip

We would need the following result which is a slightly generalized version of Proposition \ref{tau magn}. The hypothesis is inspired from Lemma 2 and the proof of the Main Theorem in \cite{krithika2025inflated}.

\begin{lemma} \label{inv tau lemma}
    Consider $M/L/K$. Consider $N/J/K$ such that $\tilde{M}\cap \tilde{N}=K$. Then we have 
\begin{enumerate}
\item $[LJ:K]=[L:K][J:K]$ and $r_K(LJ)=r_K(L)r_K(J)$ and $t_K(LJ)=t_K(L)t_K(J)$ and $\rho_K(MN,LJ)=\rho_K(M,L)\rho_K(N,J)$ and $\tau_K(MN,LJ)= \tau_K(M,L)\tau_K(N,J)$ and $\gamma_K(MJ,LJ)=\gamma_K(M,L)\gamma_K(N,J)$.

\smallskip

\item $\rho_K(MJ,LJ)=r_K(J)\rho_K(M,L)$ and $\tau_K(MJ,LJ)=[J:K]\ \tau_K(M,L)$ and $\gamma_K(MJ,LJ)=[J:K]\gamma_K(M,L)$.\smallskip

\item  $\rho_K(\tilde{L}J,LJ)=r_K(J)[L:K]$ and $\tau_K(\tilde{L}J,LJ)=[J:K]\ t_K(L)$ and $\gamma_K(\tilde{L}J,LJ)=[\tilde{L}J:K]$.

\end{enumerate}
    
\end{lemma}

\begin{proof}
   We will prove part (1). Since $\tilde{M}\cap \tilde{J}=K$, we have $\tilde{L}\cap\tilde{J}=K$. Let $G = {{\rm Gal}}(\tilde{L}/K)$, $H= {{\rm Gal}}(\tilde{L}/L)$, $S= {{\rm Gal}}(\tilde{J}/K)$ and $T = {{\rm Gal}}(\tilde{J}/J)$. Let $G'=Gal(\tilde{L}\tilde{J}/K)$ and $H'=Gal(\tilde{L}\tilde{J}/LJ)$. We can identify $G'$ with $G\times S$. Thus $H'=(H\times S)\cap (G\times T)=H\times T$ and hence $[LJ:K]=[L:K][J:K]$. Also $N_{G'}(H')= N_G(H)\times N_S(T)$. Thus $s_K(LJ)=[G':N_{G'}(H')]=s_K(L)s_K(J)$. Thus the $s_K(LJ)$ many distinct fields isomorphic to $LJ/K$ are precisely $L_iJ_j$ for $1\leq i\leq s_K(L)$ and $1\leq j\leq s_K(J)$ where $L_i$'s are all the distinct fields isomorphic to $L/K$ and $J_j$'s are all the distinct fields isomorphic to $J/K$. We also have $r_K(LJ)=r_K(L)r_K(J)$. We can also show that $H'^{G'}= H^G \times T^S$. Hence $t_K(LJ)=[G': H'^{G'}]=t_K(L) t_K(J)$.\smallskip

Since $\tilde{M}\cap \tilde{N}=K$. Hence $MN\cap \tilde{N}=N$ and $MN\cap \tilde{M}=M$.
Now suppose $L_i J_j\subset MN$. Thus $J_j\subset MN$. Hence $J_j\cap \tilde{N}\subset MN\cap \tilde{N}$. Since $J_j\subset \tilde{N}$ and $MN\cap \tilde{N}=N$. Therefore $J_j\subset N$. Similarly $L_i\subset M$. Thus the distinct fields inside $MN$ isomorphic to $LJ/K$ are precisely $\{L_iJ_j\}_{1\leq i\leq a, 1\leq j\leq b}$ where $\{L_i\}_{i=1}^a$ are fields inside $M$ isomorphic to $L/K$ and $\{J_j\}_{j=1}^b$ are fields inside $N$ isomorphic to $J/K$. Hence $\rho_K(MN,LJ)/r_K(LJ)=(\rho_K(M,L)/r_K(L))(\rho_K(N,J)/r_K(J))$. Thus $\rho_K(MN,LJ)=\rho_K(M,L)\rho_K(N,J)$. \smallskip

Now $Gal(\tilde{L}\tilde{J}/\tilde{L})=1\times S$ and $Gal(\tilde{L}\tilde{J}/\tilde{J})=G\times 1$. If $H_i= Gal(\tilde{L}/L_i)$ and $T_j= Gal(\tilde{J}/J_j)$ then $H_i\times S=Gal(\tilde{L}\tilde{J}/L_i)$ and $G\times T_j=Gal(\tilde{L}\tilde{J}/J_j)$. Thus $Gal(\tilde{L}\tilde{J}/(L_i J_j))=(H_i\times S)\cap (G\times T_j)=H_i\times T_j$. Hence $Gal(\tilde{L}\tilde{J}/(\cap_{1\leq i\leq a, 1\leq j\leq b}\ L_iJ_j))$ is generated by subgroups $\{H_i\times T_j\}_{1\leq i\leq a, 1\leq j\leq b}$ which is $H'\times T'$ where $H'$ is subgroup of $G$ generated by subgroups $\{H_i\}_{i=1}^a$ and $T'$ is subgroup of $S$ generated by subgroups $\{T_j\}_{j=1}^b$. Also $Gal(\tilde{L}\tilde{J}/(\cap_{i=1}^a L_i))=H'\times S$ and $Gal(\tilde{L}\tilde{J}/(\cap_{j=1}^b J_j))=G\times T'$. Since $(H'\times S)\cap (G\times T')=H'\times T'$. Hence $(\cap_{i=1}^a L_i)(\cap_{j=1}^b J_j)=(\cap_{1\leq i\leq a, 1\leq j\leq b}\ L_iJ_j)$. Thus $\tau_K(MJ,LJ)=\tau_K(M,L)\tau_K(N,J)$. Finally, $\Pi_{1\leq i\leq a, 1\leq j\leq b}\ L_i J_j=(\Pi_{i=1}^a\ L_i)(\Pi_{j=1}^b\ J_j)$. Thus $\gamma_K(MJ,LJ)=\gamma_K(M,L)\gamma_K(N,J)$.\end{proof}

\smallskip

\begin{lemma}
\label{lin disj}
Given $L/K$.

\begin{enumerate}
    \item If for a group $G$, direct products $G^n$ are realisable as Galois group over $K$ for each $n\in \mathbb{N}$ then we have a Galois extension $L'/K$ such that $Gal(L'/K)=G$ and $L'\cap L=K$ that is $L'$ and $L$ are linearly disjoint over $K$.

    \item If $K$ is a number field and $G$ is a solvable group then we have a Galois extension $L'/K$ such that $Gal(L'/K)=G$ and $L'\cap L=K$.

    \item  If $K$ is a number field then for any $n\in \mathbb{N}$, we have an extension $L'/K$ such that $Gal(L'/K)=\mathfrak{S}_n$ and $L'\cap L=K$.
\end{enumerate}
    
\end{lemma}

\begin{proof}\hfill

\begin{enumerate}
    \item Follows from the proof in \cite{Bhagwat_2025} of Lemma \ref{3.1.4}  (essentially the proof of Theorem 4.2 in \cite{conrad2023galois}).\smallskip

    \item Since direct product of solvable groups is solvable, the $n$-fold direct products $G^n$ for $n\in \mathbb{N}$ are solvable. By Shafarevich's theorem (\cite{Shafarevich1989}), $G^n$ for $n\in \mathbb{N}$ are realizable as Galois groups over $\mathbb{Q}$. By Lemma \ref{3.1.4} they are realisable over number field $K$. By part (1) we are done.\smallskip

    \item  The proof is similar to proof of Part (a) of Main theorem in \cite{krithika2025inflated} but we provide it nonetheless. By Lemma \ref{Sn}, there are infinitely many pairwise linearly disjoint extension $L'/K$ such that $Gal(L'/K)=\mathfrak{S}_n$. Call the infinite set of these extensions as $\mathfrak S$.

\smallskip

Since $L/K$ is separable, it has finitely many intermediate extensions. Consider the subset $\mathfrak S'$ of $\mathfrak S$ consisiting of those $L'/K$ which contain atleast one non-trivial intermediate extension of $L/K$. Due to pairwise linear disjointness, $\mathfrak S'$ is finite. Thus we are done.
    
    \end{enumerate}\end{proof}


\begin{proof}[Proof of Theorem \ref{inv tau}]\hfill

Case 1 : Suppose $\tau=t$ or $\tau=n$. Thus by Theorem \ref{t,r} for $n>2$ with $t|n$, we have an extension $L/K$ with $[L:K]=n$ and $t_K(L)=t$. Thus for $M=\tilde{L}$ we have $\tau_K(M,L)=t_K(L)=t$ and for $M=L$ we have $\tau_K(M,L)=[L:K]=n$. \smallskip

 Case 2 : Suppose $n/\tau>2$ and $\tau>2$. Then by Lemma \ref{Sn}, we have an extension $L'/K$ of degree $n/\tau$ such that its Galois closure $\tilde{L'}$ has Galois group $\mathfrak{S}_{n/\tau}$ over $K$. Thus by Theorem 7.3.6 (2) in \cite{Bhagwat_2025} we have $t_K(L')=1$ and $\tau_K(\tilde{L'},L')=t_K(L')=1$. Since we have $t|\tau$ and $\tau>2$. Thus by Theorem \ref{t,r}, we have $J/K$ with $[J:K]=\tau$ and $t_K(J)=t$. Now by Lemma \ref{lin disj} we can choose the fields such that $\tilde{L'}\cap \tilde{J}=K$.
   Let $L=L'J$ and $M=\tilde{L'}J$. Thus by Lemma \ref{inv tau lemma} we have $[L:K]=[L':K][J:K]=(n/\tau)\cdot \tau=n$ and $t_K(L)=t_K(L')t_K(J)=1\cdot t=t$ and $\tau_K(M,L)=[J:K]t_K(L')=\tau\cdot1=\tau$.\smallskip

Case 3 : Suppose $n=2\tau$ and $t>1$ and $\tau/t>2$. Then we have an extension $J/K$ of degree $\tau/t$ such that its Galois closure $\tilde{J}$ has Galois group $\mathfrak{S}_{\tau/t}$ over $K$. Thus $t_K(J)=1$. Since we have $t|2t$ and $t>1$. Thus by Theorem \ref{t,r}, we have $L'/K$ with $[L':K]=2t$ and $t_K(L')=t$. Also $\tau_K(\tilde{L'},L)=t$. Now by Lemma \ref{lin disj} we can choose the fields such that $\tilde{L'}\cap \tilde{J}=K$.
   Let $L=L'J$ and $M=\tilde{L'}J$. Thus by Lemma \ref{inv tau lemma} we have $[L:K]=[L':K][J:K]=(2t)\cdot (\tau/t)=2\tau=n$ and $t_K(L)=t_K(L')t_K(J)=t\cdot 1=t$ and $\tau_K(M,L)=[J:K]t_K(L')=(\tau/t)\cdot t=\tau$.\smallskip

   Case 4 : Suppose $n=2\tau$ and $\tau=2t$ and $t=2m$ with $m>1$. So $n=8m$ and $\tau=4m$. By Theorem \ref{t,r}, we have extension $L'/K$ with $[L':K]=4$ and $t_K(L')=2$ and extension $J/K$ with $[J:K]=2m$ and $t_K(J)=m$. Now by Lemma \ref{lin disj} we can choose the fields such that $\tilde{L'}\cap \tilde{J}=K$.
   Let $L=L'J$ and $M=\tilde{L'}J$. Thus by Lemma \ref{inv tau lemma} we have $[L:K]=[L':K][J:K]=4\cdot 2m=8m=n$ and $t_K(L)=t_K(L')t_K(J)=2\cdot m=2m=t$ and $\tau_K(M,L)=[J:K]t_K(L')=2m\cdot 2=4m=\tau$.
   
\medskip

   Case 5 : $n=8,\tau=4,t=2$. First let $K=\mathbb{Q}$. Consider the polynomial $x^8-3$ which is irreducible over $\mathbb{Q}$ having a solvable Galois group by results in \cite{jacobson1990galois}. Let $a=3^{1/8}$ be the positive real root of the polynomial and $\iota$ be a primitive $4$-th root of unity. Let $L=\mathbb{Q}(a)$ and $M=\mathbb{Q}(a,\iota)$. Thus $\tau_{\mathbb{Q}}(M,L)=[\mathbb{Q}(a)\cap \mathbb{Q}(ai):\mathbb{Q}]=[\mathbb{Q}(a^2):\mathbb{Q}]=4$. Also by Example \ref{unique F example} (1), $t_{\mathbb{Q}}(L)=[\mathbb{Q}(a^4):\mathbb{Q}]=2$. Since the Galois group of Galois closure $\tilde{L}$ over $\mathbb{Q}$ is solvable and $M\subset \tilde{L}$, we are done for any number field $K$ by Lemma \ref{lin disj} (2) and Remark \ref{tau group theoretic}. \end{proof}

   \begin{proposition}\label{excep prop}
       We don't have a positive answer for Theorem \ref{inv tau} for the case $(n,t,\tau)=(4,1,2)$ for any number field.
       
   \end{proposition}

   \begin{proof}
     Let $K$ be any number field. Supppose we have extensions $L/K$ and $M/K$ such that\\
     $[L:K]=4$, $t_K(L)=1$ and $\tau_K(M,L)=2$. The groups which can possibly occur as Galois group $G=Gal(\tilde{L}/K)$ are $\mathfrak{S}_4,\mathfrak{A}_4,D_4, C_4$ and $C_2^2$. Now since $t_K(L)\neq 4$, $G$ cannot be abelian. Thus $G\neq C_4, C_2^2$. It is easy to see that if $G=D_4$ then $t_K(L)=2$ which is not the case. Finally $G\neq \mathfrak{S}_4, \mathfrak{A}_4$ otherwise $\tau_K(M,L)\neq 2$ since $\mathfrak{S}_3$ and $\mathfrak{A}_3$ are maximal subgroups in $\mathfrak{S}_4$ and $\mathfrak{A}_4$ respectively.
   \end{proof}


\subsection{Inverse Compositum Indicium Problem}

\label{inv gamma section}
    
We begin by stating the problem.

\begin{problem}\label{comp ind problem}
 Let $K$ be a number field. Classify all pairs $(n,\gamma)$ where $n>2$ and $n \mid \gamma\mid  n!$ such that there exist extensions $L/K$ and $M/K$ such that $[L:K]=n$ and compositum indicium $\gamma_K (M,L)=\gamma$.
\end{problem}
We have a positive answer to the above problem in the following cases.\smallskip

Case 1: $(n,\gamma)$ where $n=2^m a_1a_2\cdots a_k$ with each $a_i>2$ and $m=0$ or $1$ and $\gamma=2^m b_1b_2\cdots b_k$ with each (i) $b_i=\ ^{a_i}P_j$ for $j\leq a_i$ or (ii) $b_i=a_i\phi(a_i/l)$ for $a_i$ odd and $l\mid a_i$ or (iii) $b_i=a_i\cdot r^{a-1}$ for $r>1$ and $r\mid a_i$ and $a\leq (a_i/r)$.

\begin{proof}
    Firstly we show existence of $L_i/K$ and $M_i/K$ with $M_i\subset \tilde{L_i}$ for $1\leq i\leq k$ such that $[L_i:K]=a_i$ and $\gamma_K(M_i,L_i)=b_i$. The case (i) $b_i=\ ^{a_i}P_j$ for $j\leq a_i$ follows from Example \ref{rho Sn}. The case (ii) $b_i=a_i\phi(a_i/l)$ for $l\mid a_i$ for $K=\mathbb{Q}$ follows from Proposition \ref{tau interesting eg}. Since the Galois group of Galois closure $\tilde{L_i}$ over $\mathbb{Q}$ is solvable and $M_i\subset \tilde{L_i}$, we are done for any number field $K$ by Lemma \ref{lin disj} (2) and Remark \ref{gamma group theoretic}. The case (iii) $b_i=a_i\cdot r^{a-1}$ for $r>1$ and $r\mid a_i$ and $a\leq (a_i/r)$ follows from Remarks \ref{conv countereg} and \ref{rho not divide}.\smallskip

Now by Lemma \ref{lin disj}, $L_i$'s can be chosen such that for each $1\leq t\leq k-1$ we have that $\tilde{L_1}\cdots \tilde{L_{t}}$ and $\tilde{L_{t+1}}$ are linearly disjoint over $K$. Let $L'=L_1\cdots L_k$ and $M'=M_1\cdots M_k$. Thus by Lemma \ref{inv tau lemma} (1), it follows that $[L':K]=a_1\cdots a_k$ and $\gamma_K(M',L')=\gamma_K(M_1,L_1)\cdots\gamma_K(M_k,L_k)=b_1\cdots b_k$. Thus we are done for $m=0$. Now suppose $m=1$, By Lemma \ref{lemma 3 in vanchi}, there exists an extension $F/K$ of degree $2$ such that $\tilde{L'}$ and $F$ are linearly
disjoint over $K$. Let $L=L'F$ and $M=M'F$. Thus by Lemma \ref{inv tau lemma} (1), $[L:K]=2 a_1\cdots a_k = n$ and $\gamma_K(M,L)=\gamma_K(M',L')\cdot \gamma_K(F,F)=2b_1\cdots b_k=\gamma$.\end{proof}

Case 2: $(n,\gamma)$ so that for each prime $p$, there exists $k_p\leq v_p(n)$ such that $k_p\mid (v_p(\gamma)-v_p(n))$ and $v_p(\gamma)\leq p^{(v_p(n)-k_p)}\cdot k_p+ (v_p(n)-k_p)$.

\begin{proof}

First we claim that $p\mid n\iff p\mid \gamma$. Suppose $v_p(\gamma)\geq 1$. Assume on the contrary that $v_p(n)=0$. Thus $1\leq p^{-k_p}\cdot k_p-k_p$ which is a contradiction. Thus $v_p(n)\geq 1$. Let $n_p=p^{v_p(n)}$ and $\gamma_p=p^{v_p(\gamma)}$ where $p$ is prime. If for some $p$, we have $k_p=0$ or $k_p=v_p(n)$, then $v_p(\gamma)\leq v_p(n)$. Thus $v_p(n)=v_p(\gamma)$ and $\gamma_p=n_p=\ ^{n_p}P_1$. For primes with $0< k_p< v_p(n)$, let $r_p=p^{k_p}$. So $r_p>1$ and $r_p\mid n_p$. Let $a_p=\frac{(v_p(\gamma)-v_p(n))}{k_p}+1$. Thus $\gamma_p=n_p\cdot r_p^{a_p-1}$. Also $a_p=\frac{(v_p(\gamma)-v_p(n)+k_p)}{k_p}\leq p^{(v_p(n)-k_p)}=(n_p/r_p)$. By prime factorization, $n= \Pi_{p}\ n_p$ and $\gamma= \Pi_{p}\ \gamma_p$. Observe that when $v_2(n)\geq 2$ we have $n_2\geq 4>2$, and when $v_2(n)=0$ we have $n_2=\gamma_2=1$ and when $v_2(n)=1$, we have $k_p=0$ or $1$ and so $v_2(\gamma)=v_2(n)=1$. Thus by Case 1, we are done 
\end{proof}

\begin{remark}
  One of the major difficulties in giving a complete solution to Problem \ref{comp ind problem} lies in the fact that $\gamma_K(M,L)\mid \gamma_K(\tilde{L},L)=|Gal(\tilde{L}/K)|\mid ([L:K]!)$. Given an integer $n>2$, it seems a hard problem to completely classify group orders of transitive subgroups of $\mathfrak{S}_n$ which are realizable as Galois groups over a given number field $K$.
  \end{remark}

  \begin{remark}
   The condition $n\mid \gamma\mid n!$ in Problem \ref{comp ind problem} is not sufficient in general.  We don't have a positive answer for Problem \ref{comp ind problem} for $(n,\gamma)=(5,15)$ for $K=\mathbb{Q}$. Assume on the contrary, there exist extensions $L/\mathbb{Q}$ and $M/\mathbb{Q}$ such that $[L:\mathbb{Q}]=5$ and compositum indicium $\gamma_{\mathbb{Q}} (M,L)=15$. Let $L_1/\mathbb{Q},L_2/\mathbb{Q}$ be two extensions that are isomorphic to $L/\mathbb{Q}$ such that $L_1,L_2\subset M$ and $L_1\neq L_2$. As $L_1,L_2\subset L_M$ and $[L_M:\mathbb{Q}]=15$, so $[L_M:L_1]=3$ and thus $L_M=L_1L_2$. Thus $(5,5,15)$ is compositum feasible triplet over $\mathbb{Q}$. By Theorem 36 in \cite{drungilas2012degree} (which says that $(5,5,15)$ is not compositum feasible), we have a contradiction.
  \end{remark}

\section{Applications of Root Cluster Theory}

\subsection{Minimal Generating Sets of Galois Closure of an Extension}\label{Minimal Generating Sets of Galois Closure}

We have the notion of minimal generating sets of the splitting field of a polynomial over $K$ introduced by the author and Vanchinathan in Section 2 of their work in \cite{jaiswal2025minimal}. The following is a field theoretic formulation of the same in light of Proposition 2.1 in \cite{jaiswal2025minimal}.\smallskip

Consider an extension $L/K$. Let $S= \{L_i\}_{i=1}^s$ where $L_i$'s are the $s=s_K(L)$ many distinct fields isomorphic to $L$ over $K$. For any set $B\subset S$ we denote compositum of fields in $B$ as $L_B$.

\begin{definition}
      A set $B \subset S$ is called a minimal generating set of the Galois Closure $\tilde{L}$ of $L/K$ if the following hold.
      
      \begin{enumerate}
          \item $L_B =\tilde{L}$  

          \item For any set $A\subsetneq B$, we have $L_A\neq \tilde{L}$. 
      \end{enumerate}

We also refer to this set simply as a minimal generating set for $L/K$.

\end{definition}

\begin{remark}\label{min gen group theoretic}
    Equivalently $B=\{L_{i_j}\}_{j=1}^m\subset S$ is a minimal generating set of the Galois closure of $L/K$ if $\cap_{j\in [m]}\ H_{i_j}=1$ and for any $J \subsetneq [m]$ we have $\cap_{j\in J}\ H_{i_j}\neq 1$ where $[m]=\{1,2,\dots,m\}$ and $G=Gal(\tilde{L}/K)$ and $H_{i_j}$'s are subgroups of $G$ which fix $L_{i_j}$'s respectively (which are conjugates of $Gal(\tilde{L}/L)$ in $G$).
\end{remark}

\begin{remark}\label{min gen isom}
   Suppose $B=\{L_{i_j}\}_{j=1}^m\subset S$ is a minimal generating set of the Galois closure of $L/K$. It is easy to verify that for any $\sigma\in Gal(\tilde{L}/K)$ we have that $B'=\sigma(B)=\{\sigma(L_{i_j})\}_{j=1}^m\subset S$ is also a minimal generating set for $L/K$.
\end{remark}

The following is a reformulation of Lemma 2.1 (2) in \cite{jaiswal2025minimal}.

\begin{lemma}

\label{from generating to minimal}
Let $C\subset S$ be a generating set of the Galois closure of $L/K$ i.e. $L_C=\tilde{L}$. Then there exists $B\subset C$ such that $B$ is a minimal generating set of the Galois closure of $L/K$.     

\end{lemma}

The following is a reformulation of Proposition 2.3 (1) in \cite{jaiswal2025minimal}.

\begin{lemma}\label{lem min gen}
    Consider $B=\{L_{i_j}\}_{j=1}^m\subset S$. We have that $B$ is a minimal generating set of the Galois closure of $L/K$ if and only if for every permutation $(l_1,l_2,\dots, l_m)$ of $(i_1,i_2,\dots, i_m)$, 
 \[  K \subseteq L_{l_1} \subseteq L_{l_1} L_{l_2}  \subseteq \dots \subseteq L_{l_1} L_{l_2}\cdots L_{l_m} \] is a cluster tower for $L/K$ of length $m$. \smallskip
 
 In this case we refer any such cluster tower as a minimal cluster tower.
\end{lemma}

\begin{proposition}
    
\label{SCM min gen}
    Suppose $M/K$ is obtained by strong cluster magnification from $L/K$ through $F/K$. Then $B=\{L_{i_j}\}_{j=1}^m\subset S$ is a minimal generating set of the Galois closure of $L/K$ if and only if $B'=\{L_{i_j}F\}_{j=1}^m$ is a minimal generating set of the Galois closure of $M/K$. In this case $|B|=|B'|$.
\end{proposition}

\begin{proof}

Suppose $B$ is a minimal generating set of the Galois closure of $L/K$. Thus by Lemma \ref{lem min gen}, this is equivalent to saying that for every permutation $(l_1,l_2,\dots, l_m)$ of $(i_1,i_2,\dots, i_m)$, 
 \[  K \subseteq L_{l_1} \subseteq L_{l_1} L_{l_2}  \subseteq \dots \subseteq L_{l_1} L_{l_2}\cdots L_{l_m} \] is a cluster tower for $L/K$ of length $m$. Hence by Proposition \ref{SCM cluster tower}, this is equivalent to saying that for every permutation $(l_1,l_2,\dots, l_m)$ of $(i_1,i_2,\dots, i_m)$, 
 \[  K \subseteq L_{l_1} F \subseteq L_{l_1} L_{l_2} F  \subseteq \dots \subseteq L_{l_1} L_{l_2}\cdots L_{l_m} F \] is a cluster tower for $M/K$ of length $m$. Therefore by Lemma \ref{lem min gen}, this is equivalent to $B'$ being a minimal generating set of the Galois closure of $M/K$.\end{proof}

Proposition 2.2 in \cite{jaiswal2025minimal} demonstrates that two minimal generating sets of Galois closure need not have same cardinalities. Thus we also have the following notion.

\begin{definition}
    $B\subset S$ is said to be shortest minimal generating set of the Galois closure of $L/K$ if $B$ is minimal generating set with least possible cardinality. We similarly define longest minimal generating set having greatest possible cardinality.
\end{definition}


\begin{proof}[Proof of Theorem \ref{inv min gen xn} (1)]

Let $n'=\frac{n}{2^{v_2(n)}}$. Thus $n'>2$ is an odd composite number. By Main Theorem (1) in \cite{jaiswal2025minimal}, there exists an $L'/\mathbb{Q}$ of degree $n'$ for which the Galois closure has minimal generating sets of cardinalities $2,3,\dots , \omega(n')$ (where $\omega(n')$ is the number of distinct prime divisors of $n'$) and these are the only possible cardinalities for minimal generating set for $L'/K$. Observe that $\omega'(n)=\omega(n')$.\smallskip

Now from the proof in \cite{jaiswal2025minimal}, we have that the Galois group of Galois closure of $L'/\mathbb{Q}$ is $G=\mathbb{Z}/n\mathbb{Z}\rtimes (\mathbb{Z}/n\mathbb{Z})^{\times}$ which is solvable. Also $H=Gal(\tilde{L'}/L')=1\times (\mathbb{Z}/n\mathbb{Z})^{\times}$. By Lemma \ref{lin disj} (2), $G$ is realizable as a Galois group over $K$ say for $E/K$. Let $M/K$ be fixed field of $H$ inside $E/K$. By Remark \ref{min gen group theoretic}, $M/K$ has degree $n'$ and its Galois closure has minimal generating sets of cardinalities $2,3,\dots , \omega(n')$ and these are the only possible cardinalities for minimal generating set for $M/K$.\smallskip

By Lemma \ref{lemma 3 in vanchi}, there exists a cyclic Galois extension $F/K$ of degree $2^{v_2(n)}$ such that $\tilde{M}$ and $F$ are linearly
disjoint over $K$. Let $L=MF$. By Proposition \ref{cluster magnification}, $[L:K]=n'\cdot 2^{v_2(n)}=n$. By Proposition \ref{SCM min gen}, the Galois closure of $L/K$ has minimal generating sets of cardinalities $2,3,\dots , \omega'(n)$ and these are the only possible cardinalities for minimal generating set for $L/K$. Also observe that $E=\tilde{M}$ and $\tilde{L}=\tilde{M}F=EF$ and $Gal(\tilde{L}/K)\cong Gal(E/K)\times Gal(F/K)\cong G\times (\mathbb{Z}/(2^{v_2(n)})\mathbb{Z})$ which is solvable. 
 \end{proof}

An improvement on Theorem \ref{inv min gen xn} (1) (Similar improvement as in Theorem \ref{inverse cluster size improv}): Let $K$ be a number field and $n$ be an integer such that $\frac{n}{2^{v_2(n)}}>2$ is composite. Then we get arbitrarily large finite families of degree $n$ extensions $L/K$  inside $\bar{K}$ which are pairwise non-isomorphic over $K$ and are pairwise linearly disjoint over $K$, for which the Galois closure has minimal generating sets of cardinalities $2,3,\dots , \omega'(n)$ and these are the only possible cardinalities for minimal generating set of Galois closure of $L/K$.

    
\smallskip


\begin{proof}[Proof of Theorem \ref{inv min gen xn} (2)]
    Consider any number field $K$ and any $n>2$ and $d\mid n$ with $d>2$. By Main Theorem (2) in \cite{jaiswal2025minimal}, there exists an $L'/K$ of degree $d$ for which the Galois closure has all its minimal generating sets of cardinality $k$ for the following values of $k$ : (i) $k=d-1$, (ii) $k=d-2$. By Lemma \ref{lemma 3 in vanchi}, there exists a Galois extension $F/K$ of degree $n/d$ such that $\tilde{L'}$ and $F$ are linearly
disjoint over $K$. Let $L=L'F$. By Proposition \ref{cluster magnification}, $[L:K]=d\cdot (n/d)=n$. By Proposition \ref{SCM min gen}, the Galois closure of $L/K$ has all its minimal generating sets of cardinality $k$.\smallskip

Now consider any number field $K$ and any $n>2$ and $k=2$. Let $q>2$ be a prime power such that $q\mid n$. By Main Theorem (3) in \cite{jaiswal2025minimal}, there exists an $L'/\mathbb{Q}$ of degree $q$ for which the Galois closure has all its minimal generating sets of cardinality $2$. Now from the proof in \cite{jaiswal2025minimal}, we have that the Galois group of Galois closure of $L'/\mathbb{Q}$ is $G=\mathbb{F}_q\rtimes (\mathbb{F}_q)^{\times}$ which is solvable. Let $H=Gal(\tilde{L'}/L')$. By Lemma \ref{lin disj} (2), $G$ is realizable as a Galois group over $K$ say for $E/K$. Let $M/K$ be fixed field of $H$ inside $E/K$. By Remark \ref{min gen group theoretic}, $M/K$ has degree $q$ and its Galois closure has all its minimal generating sets of cardinality $2$. By Lemma \ref{lemma 3 in vanchi}, there exists a Galois extension $F/K$ of degree $n/q$ such that $\tilde{M}$ and $F$ are linearly
disjoint over $K$. Let $L=MF$. Thus $[L:K]=n$ and by Proposition \ref{SCM min gen}, the Galois closure of $L/K$ has all its minimal generating sets of cardinality $2$.\end{proof}

\begin{proof}[Proof of Theorem \ref{inv min gen xn} (3)]
By Main Theorem (3) in \cite{jaiswal2025minimal}, there exists an $L'/\mathbb{Q}$ of degree $n'$ for which the Galois closure has all its minimal generating sets of cardinality $k$ for the following values of $n'$ and $k$:\begin{enumerate}  
    
\item (i) $n'=12$ and $k=5$, (ii) $n'=11$ and $k=4$.

\item $n'=p+1$ where $p$ is an odd prime and $k=3$.

\end{enumerate}\smallskip

Now the result follows by applying Lemma \ref{lemma 3 in vanchi} and Proposition \ref{SCM min gen}.
\end{proof}

\begin{proof}[Proof of Theorem \ref{inverse min gen}]

Let $r=n/s$. Hence $r>1$. By the proof of second approach of Theorem \ref{inv root cap} we have $L/K$ such that $[L:K]=n$ and $r_K(L)=r$. Thus $s_K(L)=s$. Let $L_i$'s be the $s$ many distinct fields isomorphic to $L$ over $K$. Now by Remark \ref{rem for min gen set} we have for every permutation $(i_1,i_2,\dots, i_s)$ of $(1,2,\dots, s)$ that \[ K \subseteq L_{i_1} \subseteq L_{i_1} L_{i_2} \subseteq \dots \subseteq L_{i_1} L_{i_2} \cdots L_{i_s}\] is a cluster tower for $L/K$ of length $s$. Thus by Lemma \ref{lem min gen}, $S=\{L_i\}_{i=1}^s$ is a minimal generating set of the Galois closure of $L/K$ which is the unique minimal generating set and thus also a shortest minimal generating set.\end{proof}

\begin{remark}
   Assuming the condition $s\neq n$ is necessary in Theorem \ref{inverse min gen}. Suppose $L/K$ is degree $n$ extension. Then we can't have a minimal generating set of cardinality $n$. This follows from the argument in Remark \ref{remark inv root cap} as cardinality of minimal generating set being $n$ forces $s_K(L)=n$ and $r_K(L)=1$. Any subset of minimal generating set of cardinality $n-1$ also generates the Galois closure which contradicts the minimality.
\end{remark}

An improvement on Theorem \ref{inverse min gen} (Similar improvement as in Theorem \ref{inverse cluster size improv}): Let $K$ be a number field. Given positive integers $n>2$ and $s|n$ with $s<n$. We get arbitrarily large finite families of degree $n$ extensions $L/K$  inside $\bar{K}$ which are pairwise non-isomorphic over $K$ and are pairwise linearly disjoint over $K$, for which the Galois closure has a minimal generating set of cardinality $s$. Furthermore each $L/K$ satisfies $s_K(L)=s$. Hence there is a unique minimal generating set for the Galois closure of $L/K$ which is thus, also a shortest minimal generating set.

\begin{proposition}

    \label{min gen unique prop}

    Consider $L/K$. Let $D=\{L_i\ |\ L_{S\backslash \{L_i\}}\neq \tilde{L}\}$.
    
\begin{enumerate}

\item  Let $B$ be any minimal generating set of the Galois closure of $L/K$. Then $B\supset D$.

\item We have that either $D=\emptyset$ or $D=S$. In the second case $S$ is the only minimal generating set of the Galois closure of $L/K$.

\item Suppose $B\subset S$ is the unique minimal generating set of the Galois closure of $L/K$. Then $B=S$.
 \end{enumerate}

    

    
\end{proposition}

\begin{proof}\hfill

\begin{enumerate}
    \item Let $L_i$ be such that $L_{S\backslash \{L_i\}}\neq \tilde{L}$. Assume on the contrary that $L_i\not \in B$. Thus $B\subset S\backslash \{L_i\}$. Thus $L_B\subset L_{S\backslash \{L_i\}}$. Hence $L_B\neq \tilde{L}$ which is a contradiction.\smallskip

    \item Suppose $D\neq \emptyset$ i.e. there is an $L_i$ such that $L_{S\backslash \{L_i\}}\neq \tilde{L}$. Now for any $1\leq j\leq s=s_K(L)$, we can choose $\sigma\in Gal(\tilde{L}/K)$ such that $\sigma(L_i)=L_j$. Thus $\sigma(S\backslash \{L_i\})=S\backslash \{L_j\}$. Hence $L_{S\backslash \{L_j\}}=\sigma (L_{S\backslash \{L_i\}})\neq \tilde{L}$. Thus $L_j\in D$ for any $1\leq j \leq s$.

    \smallskip

   \item Consider $L_i$ such that $L_{S\backslash \{L_i\}}=\tilde{L}$. Now since $S\backslash \{L_i\}$ is a generating set of $\tilde{L}$, by Lemma \ref{from generating to minimal} we have a minimal generating set $B'$ of $\tilde{L}$ such that $B'\subset S\backslash \{L_i\}$. By uniqueness of $B$ we have $B=B'$. Thus $B\subset S\backslash \{L_i\}$. Hence $L_i\not \in B$. By part (1), $B=D$. By part (2) we are done. Alternatively, let $L_i\in B$. Now for any $1\leq j\leq s=s_K(L)$, we can choose $\sigma\in Gal(\tilde{L}/K)$ such that $\sigma(L_i)=L_j$. By Remark \ref{min gen isom}, we have that $B'=\sigma(B) \subset S$ is also a minimal generating set of $\tilde{L}$. By uniqueness of $B$ we have $B=B'$. Since $L_j\in B'$. Thus $L_j\in B$ for any $1\leq j\leq s$. Therefore $B=S$.
\end{enumerate}\end{proof}






\subsection{Compositum Feasible Triplets} \label{Compositum Feasible Triplets}

Along with the notion of compositum feasible triplets, the following notions were also introduced by Drungilas et al. in \cite{drungilas2012degree} with $\mathbb{Q}$ as the base field $K$.

\begin{definition}
 Let $K$ be a perfect field. A triplet of positive integers $(a,b,c)$ is said to be 
    
    \begin{itemize}

        \item sum feasible over $K$ if there exist $\alpha,\beta,\gamma\in \bar{K}$ with degrees of respective minimal polynomial over $K$ as $a,b$ and $c$ such that $\alpha+\beta+\gamma=0$.

        \item product feasible over $K$ if there exist $\alpha,\beta$ and $\gamma\in \bar{K}$ with degrees of respective minimal polynomial over $K$ as $a,b$ and $c$ such that $\alpha\beta\gamma=1$.

    \end{itemize}
\end{definition}

\begin{remark}
    If $(a,b,c)$ is compositum feasible over $K$ then clearly the permutation $(b,a,c)$ is also compositum feasible over $K$. If $(a,b,c)$ is sum feasible (or product feasible) over $K$ then clearly all the six permutations are also sum feasible (or product feasible) over $K$.

\end{remark}

The proof of Proposition 1 in \cite{drungilas2012degree} also works for any infinite perfect field as base field in place of $\mathbb{Q}$. We state that result for infinite perfect fields below.

\begin{lemma} \label{comp imply sum prod} Let $K$ be an infinite perfect field. If the triplet $(a, b, c)$
is compositum feasible over $K$, then
it is also sum feasible and product feasible over $K$.
    
\end{lemma}

\begin{lemma}[part of Lemma 14 \cite{drungilas2012degree}] \label{comp condition} Suppose that $(a,b,c)$ is compositum feasible over $K$. Then $c=lcm(a,b)\cdot t$ for some positive integer $t\leq gcd(a,b)$.
\end{lemma}

\begin{remark}
    If $(a,b,c)$ is compositum feasible over $K$ and some permutation other than $(b,a,c)$ is also compositum feasible over $K$ then by Lemma \ref{comp condition}, we have $c=a$ or $c=b$. If all permutations are compositum feasible over $K$ then $a=b=c$. 
\end{remark}

\begin{proposition}
    Suppose $(a,b,c)$ is compositum feasible over $K$.

    \begin{enumerate}
        \item Suppose $a<c$ and $b<c$. Then the Galois group of the Galois closure of the compositum field over $K$ cannot be $S_c$ or $A_c$.\smallskip

        \item Suppose the Galois group of the Galois closure of the compositum field over $K$ is cyclic. Then $c=lcm(a,b)$.

     \end{enumerate}
\end{proposition}

\begin{proof}\hfill

\begin{enumerate}
    \item 

Assume on the contrary that we have $L/K$ and $L'/K$ of degrees $a$ and $b$ such that compositum $LL'/K$ has degree $c$ and Galois closure of $LL'/K$ say $M/K$ has Galois group $S_c$ or $A_c$. In that case, the subgroup fixing $LL'$ is isomorphic to $S_{c-1}$ or $A_{c-1}$ respectively. Since this subgroup is maximal in the whole group, we cannot have nontrivial intermediate fields of $LL'/K$, which is a contradiction.\smallskip

\item  Suppose we have $L/K$ and $L'/K$ of degrees $a$ and $b$ such that compositum $LL'/K$ has degree $c$ and Galois closure of $LL'/K$ say $M/K$ has cyclic Galois group. Then clearly $M=LL'$. Now $Gal(M/K)=\mathbb{Z}/c\mathbb{Z}=<\bar{1}>$. Thus $Gal(M/L)=\mathbb{Z}/(c/a)\mathbb{Z}=<\bar{a}>$ and $Gal(M/L')=\mathbb{Z}/(c/b)\mathbb{Z}=<\bar{b}>$. Since $M=LL'$, we have $<\bar{a}>\cap <\bar{b}>=\{\bar{0}\}$. We have $c\geq lcm(a,b)$. Assume $c >lcm(a,b)$. Then $\bar{0}\neq \bar{lcm(a,b)}\in <\bar{a}>\cap <\bar{b}>$ which is a contradiction. Hence $c=lcm(a,b)$.

\end{enumerate}
\end{proof}

The following is Theorem 2.4 in \cite{bhagwat2024right} stated here for $l=2$ and base field any perfect field $K$ in place of $\mathbb{Q}$ (the same proof still works). The second condition in the hypothesis is stated as in Remark 2.4.1 (2) in \cite{bhagwat2024right}.

\begin{lemma}\label{right splitting lemma}

  Let $K$ be a perfect field and $J/K$ be a  finite extension and let $E_1$ and $E_2$ be finite Galois extensions of $K$ contained in $\bar{K}$. Let $G_i=Gal(E_i/K)$ for $i=1,2$. Suppose  \smallskip
     \begin{enumerate}  
     
     \item $E_1 \cap  E_2 = J$ \smallskip

     \item for every $i$, $G_i \cong H_i \rtimes G_i/H_i$ (for some semidirect product group law) where $H_i=Gal(E_i/J)$. \smallskip
     
     \end{enumerate}
     
     Then $Gal(E_1E_2/K)\cong (H_1\times H_2)\rtimes G_1/H_1$ (for some semidirect product group law). 
    
\end{lemma}







\begin{proof}[Proof of Theorem \ref{comp feas thm} (1)]

Since $t\ |\ gcd(a,b)$, we have $gcd(a,b)=st$ for some positive integer $s$. Thus $lcm(a,b)\cdot t=ab/s$. Thus the triplet is $(a,b,ab/s)$ for $s\ |\ gcd(a,b)$.\smallskip


If $s=a$ or $b$. Then the triplet is $(a,b,b)$ with $a|b$ or $(a,b,a)$ with $b|a$. Without loss of generality let $a\leq b$ and $a|b$. By Lemma \ref{lemma 3 in vanchi} for $K/K$, we get a cyclic Galois extension $L'/K$ of degree $b$. Since $a|b$ and $L'/K$ is cyclic, we have an intermediate extension $L/K$ of degree $a$. Thus clearly $LL'=L'$ which is Galois over $K$ with a cyclic hence a solvable Galois group.

\smallskip

If $s<a,b$. Then $a,b>1$. Let $r=a/s$. Thus $r>1$. For $a=2$ we have $s=1$ and $r=2$. By Lemma \ref{lemma 3 in vanchi} for $K/K$, we get a cyclic Galois extension $L/K$ of degree $2$ and hence $r_K(L)=2$ and $G_1=Gal(\tilde{L}/K)$ is $\mathbb{Z}/2\mathbb{Z}$. For $a>2$, we proceed as we did in the proof of Theorem \ref{n,r} in \cite{Bhagwat_2025}. Thus we have that there exists an extension $L/K$ with Galois closure $\tilde{L}$ such that $[L:K]=a$, $r_K(L)=r$ and $G_1=Gal(\tilde{L}/K)$ is $(\Z/r\Z)^s \rtimes \Z / s\Z$. By Theorem 7.3.4 in \cite{Bhagwat_2025}, we get an intermediate extension $N/K$ of $L/K$ such that $L/N$ is cyclic Galois extension of degree $r$ and $N/K$ is cyclic Galois extension of degree $s$ and $H_1=Gal(\tilde{L}/N)$ is $(\Z/r\Z)^s$. Now $b/s>1$. Hence by Lemma \ref{lemma 3 in vanchi} for $\tilde{L}/K$, there exists a cyclic Galois extension $F/K$ of degree $b/s$ such that $\tilde{L}$ and $F$ are linearly
disjoint over $K$. Let $L'=NF$. Since $\tilde{L}\cap F=K$. Thus $N\cap F=K$. Hence $[L':N]=[F:K]=b/s$ and $L'/K$ is Galois with $[L':K]=[L':N][N:K]=(b/s)\cdot s=b$ and $G_2=Gal(L'/K)=(\mathbb{Z}/(b/s)\mathbb{Z})\times (\mathbb{Z}/s\mathbb{Z})$ and $H_2=Gal(L'/N)=\mathbb{Z}/(b/s)\mathbb{Z}$. We also have $\tilde{L}\cap L'=N$. Hence $[LL':N]=[L:N][L':N]=r\cdot b/s=ab/s^2$. Thus $[LL':K]=[LL':N][N:K]=(ab/s^2)\cdot s= ab/s$. By Lemma \ref{right splitting lemma} for $J=N$, $E_1=\tilde{L}$ and $E_2=L'$, we get that $Gal(\tilde{L}L'/K)$ is $(H_1\times H_2)\rtimes G_1/H_1=((\Z/(a/s)\Z)^s\times \mathbb{Z}/(b/s)\mathbb{Z})\rtimes \Z / s\Z$ (for some semidirect product group law). This group is clearly solvable. Since $\tilde{L}L'$ is Galois closure of $LL'/K$, we are done. The last assertion follows from Lemma \ref{comp imply sum prod}.\end{proof}

\begin{remark}
  Theorem \ref{comp feas thm} (1) doesn't cover the case $(a,b,lcm(a,b)\cdot t)$ where $t\leq gcd(a,b)$ and $t \nmid gcd(a,b)$. In fact it is not necessary that such a triplet is compositum feasible over number fields. For example $(p,p,p(p-l))$ is not compositum feasible over $\mathbb{Q}$ for every integer $l\geq 2$ and every prime $p>l^2-l+1$ (See Theorem 1.2 \cite{drungilas2013degree}). But we do have examples of such triplets like $(n,n,n(n-1))$ for $n\geq 3$ which are indeed compositum feasible over $\mathbb{Q}$ (See Proposition 29 (1) \cite{drungilas2012degree}).
\end{remark}

\begin{remark}
Given any compositum feasible triplet $(a,b,c)$ over a perfect field $K$, it is is not always necessary that we can choose the field extensions such that compositum has Galois closure over $K$ having a solvable Galois group. As noted in \cite{virbalas2023compositum} (See note after Lemmas 2.7 and 2.8 in \cite{virbalas2023compositum}), we have for $p$ prime and $1\leq s\leq p-1$ that if $(p,p,ps)$ is compositum feasible over $\mathbb{Q}$ for field extensions with Galois group of Galois closure of compositum field as solvable then we should have $s\ |\ (p-1)$. We have from Theorem 1.1 (c) in \cite{virbalas2023compositum} that $(11,11,11\cdot 6)$ is compositum feasible over $\mathbb{Q}$. Since $6\nmid 10$. Hence we cannot choose the field extensions for the triplet $(11,11,11\cdot 6)$ such that the Galois closure of the compositum has a solvable Galois group.

\end{remark}

 The following cases (some of which may be covered over $\mathbb{Q}$ in earlier works on this problem) follow easily (over any number field) from Theorem \ref{comp feas thm} (1).

\begin{corollary}
    We have the following triplets to be compositum feasible over any number field.

\begin{enumerate}

    \item $(n,n,nd)$ for positive integers $n$ and $d|n$. In particular $(p^k,p^k,p^j)$ for positive integers $p,k$ and $j$ where $p$ is prime and $k\leq j\leq 2k$.
    
    \smallskip

    \item $(p,t,t)$ for positive integers $p$ and $t$ where $p$ is prime and $p|t$ (\cite{drungilas2012degree}).\smallskip

    \item $(a(abc)^n,b(abc)^n, c(abc)^{n+1})$ for positive integers $a,b$ and $c$ and $n\geq 2$ (\cite{drungilas2012degree}).
\end{enumerate}
    
\end{corollary}


An improvement on Theorem \ref{comp feas thm} (1) (Similar improvement as in Theorem \ref{inverse cluster size improv}): Let $K$ be a number field. Given $(a,b,c)\in \mathcal{C}'_K$, we get arbitrarily large finite families of extensions $M/K$  inside $\bar{K}$ which are pairwise non-isomorphic over $K$ and are pairwise linearly disjoint over $K$ and each having degree $c$ and each is compositum of two intermediate extensions of degrees $a$ and $b$.
\smallskip

The following follows from Lemma \ref{3.1.4}.

\begin{proposition}\label{comp feas extend}
    Suppose $(a, b, c)$ is compositum feasible over $\mathbb{Q}$. If we can choose the field extensions such that compositum has Galois closure over $\mathbb{Q}$ having Galois group $G$ such that for every positive integer $n$, the $n$-fold direct product $G^n$ occurs as a Galois group over $\mathbb{Q}$. Then $(a, b,c)$ is compositum feasible over any number field.
\end{proposition}

\begin{corollary}\label{comp feas extend corollary}
     Suppose $(a, b, c)$ is compositum feasible over $\mathbb{Q}$. If we can choose the field extensions such that compositum has Galois closure over $\mathbb{Q}$ having a solvable Galois group. Then $(a, b,c)$ is compositum feasible over any number field.
\end{corollary}



\begin{remark}
    Clearly if $(a,b,c),(a',b',c')\in \mathcal{C}_K$ with $gcd(c,c')=1$. Then $(aa',bb',cc')\in \mathcal{C}_K$. Suppose $L/K$ and $L'/K$ have degrees $a$ and $b$ such that $LL'/K$ has degree $c$ and $M/K$ and $M'/K$ have degrees $a'$ and $b'$ such that $MM'/K$ is of degree $c'$. Since $gcd(c,c')=1$ and $lcm(a,b)\ |\ c$ and $lcm(a',b')\ |\ c'$. Thus $gcd(a,a')=gcd(b,b')=1$. Thus degrees of $LM/K$, $L'M'/K$ and $(LL')(L'M')/K$ are $aa',bb'$ and $cc'$ respectively and $(LL')(MM')=(LM)(L'M')$. 
\end{remark}

\begin{lemma}[Proposition 3.2 \cite{drungilas2016degrees}] \label{igp lemma} Suppose $(a, b, c),(a',b',c')\in \mathcal{C}_K$. If for one of the triplets we can choose the field extensions such that compositum has Galois closure over $K$ having Galois group $G$ such that for every positive integer $n$, the $n$-fold direct product $G^n$ occurs as a Galois group over $K$. Then $(aa', bb',cc')\in \mathcal{C}_K$.
    
\end{lemma}

\begin{corollary}\label{igp lemma corollary}
  Let $K$ be a number field. Suppose $(a, b, c),(a',b',c')\in \mathcal{C}_K$. If for one of the triplets we can choose the field extensions such that compositum has Galois closure over $K$ having a solvable Galois group. Then $(aa', bb',cc')\in \mathcal{C}_K$.
\end{corollary}

    




\begin{proof}[Proof of Theorem \ref{comp feas thm} (2)]  Follows from Theorem \ref{comp feas thm} (1) and Corollary \ref{igp lemma corollary}. \end{proof}

   
   

    \begin{proof}[Proof of Theorem \ref{comp feas thm} (3)]
  
  Suppose $(a,b,c),(a',b',c')\in \mathcal{C}'_K$. Thus $c=ab/s$ where $s\ |\ gcd(a,b)$ and $c'=a'b'/s'$ where $s'\ |\ gcd(a',b')$ as in proof of Theorem \ref{comp feas thm} (1). Consider $(aa',bb',cc')$. Now $cc'=(aa')(bb')/(ss')$. Since we have $ss'\ |\ gcd(a,b)gcd(a',b')$ and $gcd(a,b)gcd(a',b')\ |\ gcd(aa',bb')$, we are done.\end{proof}

   \begin{lemma}[Theorem 1.1 \cite{drungilas2016degrees}] \label{sum imply prod} If the triplet $(a, b, c)$
is sum feasible over $\mathbb{Q}$, then
it is also product feasible over $\mathbb{Q}$.
       
   \end{lemma}

   \begin{proposition}\label{Sn comp feas}
Let $K$ be a number field. Let $n\in \mathbb{N}$ and $i,j$ and $k$ be non-negative integers. 
   \begin{enumerate}
       \item   Suppose $i\leq j\leq k$ and $k+i\leq n-1$. Then $(^n P_j, ^n P_k, ^n P_{k+i})\in \mathcal{C}_K$. We have $(^n P_j, ^n P_k, ^n P_{k+i})\in \mathcal{C}'_K$ if and only if $^{n-k} P_i\ |\ ^n P_j$. In particular, for $n\geq 3$ we have $(n,n,n(n-1))\in \mathcal{C}_K\backslash \mathcal{C}'_K$.

       \item  $(^n C_j, ^n C_k, ^n C_{k+j})$ is sum feasible over $K$ for $k+j<n-1$. It is also product feasible over $\mathbb{Q}$. In particular, for $n>3$ we have that $(n,n,n(n-1)/2)$ is sum feasible over $K$.
   \end{enumerate}
       
   \end{proposition}

   \begin{proof}

   Let $f$ over $K$ be irreducible of degree $n>2$ with Galois group ${\mathfrak S}_n$ (Such a polynomial exists by Lemma \ref{Sn}).\smallskip

   \begin{enumerate}
       \item Consider $k+i$ many distinct roots of $f$ in a fixed order. Let $L$ be field generated by first $j$ many roots and $L'$ be field generated by last $k$ many roots. Then $LL'$ is field generated by all $k+i$ many roots. We know that degree of the field generated by $m$ many distinct roots over $K$ is $^nP_m$. Thus $(^n P_j, ^n P_k, ^n P_{k+i})\in \mathcal{C}_K$. Now $(^n P_j, ^n P_k, ^n P_{k+i})\in \mathcal{C}'_K$ $\iff$ $( ^n P_{k+i}/ ^n P_k)=\ ^{n-k}P_i$ divides $^n P_j$.\smallskip

       \item Consider $k+j$ many distinct roots of $f$ in a fixed order. Let $\alpha$ be sum of first $j$ many roots and $\beta$ be sum of last $k$ many roots. Let $\gamma$ be sum of all $k+j$ many roots. We know by proof of Theorem 7.3.6 in \cite{Bhagwat_2025} that degree of the field generated by sum of $m<n-1$ many distinct roots over $K$ is $^nC_m$. Thus $(^n C_j, ^n C_k, ^n C_{k+j})$ is sum feasible over $K$. By Lemma \ref{sum imply prod} product feasibility over $\mathbb{Q}$ follows.

   \end{enumerate}
       
   \end{proof}

   \begin{remark} Let $K$ be a number field. If $(a,b,c)\in \mathcal{C}'_K$ and $(a',b',c')\in \mathcal{C}_K\backslash \mathcal{C}'_K$ then its neither necessary that $(aa',bb',cc')$ lies in $\mathcal{C}'_K$ nor necessary that it lies in $\mathcal{C}_K\backslash \mathcal{C}'_K$. Let $n\geq 3$. Now we have $(n,n,n(n-1))\in \mathcal{C}_K \backslash \mathcal{C}'_K$ and $(n,n,n), (n-1,n-1,n-1)\in \mathcal{C}'_K$. Observe that $(n^2,n^2,n^2(n-1))\in \mathcal{C}_K \backslash \mathcal{C}'_K$ but on the other hand $(n(n-1),n(n-1),n(n-1)^2)\in \mathcal{C}'_K$.\smallskip

   In fact if $(a,b,c)\in \mathcal{C}_K\backslash \mathcal{C}'_K$ then there exist infinitely many elements of $\mathcal{C}'_K$ namely $(cn,cn,cn)$ for $n\in \mathbb{N}$ such that $(a\cdot cn,b\cdot cn,c\cdot cn)\in \mathcal{C}'_K$. Similarly there exist infinitely many elements of $\mathcal{C}'_K$ namely $(p,p,p)$ for $p>c$ prime such that $(a\cdot p, b\cdot p, c\cdot p)\in \mathcal{C}_K\backslash \mathcal{C}'_K$.\smallskip

   Similarly if $(a,b,c),(a',b',c')\in \mathcal{C}_K\backslash \mathcal{C}'_K$ such that $(aa',bb',cc')\in \mathcal{C}_K$. Then its neither necessary that $(aa',bb',cc')$ lies in $\mathcal{C}'_K$ nor necessary that it lies in $\mathcal{C}_K\backslash \mathcal{C}'_K$. For $K=\mathbb{Q}$, we have $(3,3,6),(4,4,12)\in \mathcal{C}_K\backslash \mathcal{C}'_K$. Clearly $(12,12,72)\in \mathcal{C}'_K$. By Theorem 1 in \cite{drungilas2019degree}, we have that $(9,9,36)\in \mathcal{C}_K$ but clearly $(9,9,36)\not \in \mathcal{C}'_K$. 
\end{remark}

The following notion was introduced in Appendix A in \cite{maciulevivcius2023degree}.

\begin{definition}
    A triplet $(a,b,c)\in \mathcal{C}_K$ is called reducible if it can be written as product of two nontrivial $\mathcal{C}_K$-triplets that is $(a,b,c)=$
    $(a_1a_2,b_1b_2,c_1c_2)$ where $(a_1,b_1,c_1),(a_2,b_2,c_2)\in \mathcal{C}_K\backslash \{(1,1,1)\}$. It is called irreducible if it is not reducible.
\end{definition}



\begin{proof}[Proof of Theorem \ref{comp feas thm} (4)]  We will prove the result step by step. Suppose $(a,b,c)$ is irreducible.

\begin{enumerate}
\item Then we have $t=1$.\smallskip 

Suppose $t\neq 1$. Thus $a,b\neq 1$ and $(a,b,c)=(a\cdot 1, (b/t) \cdot t, lcm(a,b)\cdot t)$. Clearly $(1,t,t)\in \mathcal{C}'_K\backslash \{(1,1,1)\}$. Now consider $(a,b/t,lcm(a,b))$. Since $t\ |\ gcd(a,b)$. Thus $lcm(a,b)\cdot t\ |\ ab$. Hence we have $lcm(a,b/t)\ |\  lcm(a,b)\ |\ a\cdot (b/t)$. Thus $(a,b/t,lcm(a,b))\in \mathcal{C}'_K\backslash \{(1,1,1)\}$.\smallskip

\item Suppose $a,b>1$, then for any prime $p$, we have $p|a$ if and only if $p|b$.\smallskip

 From part (1), $c=lcm(a,b)$. It is enough to show that for any prime $p$, we have $p|a$ implies $p|b$. Assume on the contrary that there is a prime such that $p|a$ but $p\nmid b$. Observe that $(a,b,c)=((a/p)\cdot p,b\cdot 1,(lcm(a,b)/p)\cdot p)$. Clearly $(p,1,p)\in \mathcal{C}'_K\backslash \{(1,1,1)\}$. Since $p|a$ but $p\nmid b$, we have $lcm(a/p, b)=lcm(a,b)/p$. Also $b>1$. Thus $(a/p, b, lcm(a,b)/p)\in \mathcal{C}'_K\backslash \{(1,1,1)\}$.\smallskip

 \item Suppose $a,b>1$, then $a$ is a prime power.\smallskip
 
 Suppose $a$ has more than one prime in its factorization. Let $a=p^{v_p(a)}a'$ where $v_p(a)\geq 1$ and $a'>1$. From part (2) we have $b=p^{v_p(b)}b'$ where $v_p(b)\geq 1$ and $b'>1$. Now $lcm(a,b)=p^{max\{v_p(a),v_p(b)\}}lcm(a',b')$. Thus $(a,b,c)=(p^{v_p(a)}\cdot a', p^{v_p(b)}\cdot b', p^{max\{v_p(a),v_p(b)\}}\cdot lcm(a',b'))$. Clearly $(p^{v_p(a)}, p^{v_p(b)}, p^{max\{v_p(a),v_p(b)\}}),(a',b',lcm(a',b'))\in \mathcal{C}'_K\backslash \{(1,1,1)\}$.\smallskip

 \item  Suppose $a,b>1$, then $(a,b,c)=(p,p,p)$ for some prime $p$.\smallskip
 
 By parts (2) and (3) we have $(a,b,c)=(p^i,p^j,p^{max\{i,j\}})$. Without loss of generality let $i\leq j$. Then $(a,b,c)=(p^i\cdot 1,p^i\cdot p^{j-i},p^i\cdot p^{j-i})$. Thus $i=0$ or $j=i$. If $i=0$ then by part (2) we have $j=0$ which is a contradiction. If $j=i$ then $(a,b,c)=(p^i,p^i,p^i)$. We must have $i=1$ because if $i>1$ then $(a,b,c)=(p\cdot p^{i-1},p\cdot p^{i-1},p\cdot p^{i-1})$ which is a contradiction.\smallskip

 \item Finally if $a=1$ or $b=1$ then $(a,b,c)=(1,1,1)$ or $(1,p,p)$ or $(p,1,p)$ for some prime $p$.\smallskip

 It is enough to show for $a=1$ and $b>1$. If $b$ is not prime then $b=de$ where $d,e>1$. Then $(a,b,c)=(1\cdot 1,d\cdot e,d\cdot e)$ which is a contradiction.\end{enumerate}\end{proof}

\begin{remark}\hfill

\begin{enumerate}
\item The above proof also demonstrates that whenever $(a,b,c)\in \mathcal{C}'_K$ is reducible inside $\mathcal{C}_K$, it is also reducible inside $\mathcal{C}'_K$ (product of two triplets in $\mathcal{C}'_K$).

\smallskip

\item If we do not have the condition that $t\ |\ gcd(a,b)$, then it is not necessary that $(a,b,c)$ being irreducible implies $t=1$. For instance the triplet $(n,n,n(n-1))$ for $n\geq 3$ which is compositum feasible over $\mathbb{Q}$ is irreducible as shown in Proposition A1 in \cite{maciulevivcius2023degree}.

   \end{enumerate}
\end{remark}

\begin{proposition}\label{xn comp feas}
Let $K$ be a number field. Then

\begin{enumerate}
    \item For any $n\in \mathbb{N}$, $(n,n,n\phi(n))\in \mathcal{C}_K$ where $\phi$ is the Euler totient function. Furthermore we can choose the field extensions such that compositum has Galois closure over $K$ having a solvable Galois group. 
    
    \item  $(n,n,n\phi(n))\in \mathcal{C}'_K$ if and only if $n=1$ or $n=2^l 3^m$ where $l\geq 1$ and $m\geq 0$. 

    \item $(n,n,n\phi(n))$ is irreducible if and only if $n=1$ or $n$ is a prime.

\end{enumerate}
\end{proposition}

\begin{proof}\hfill

\begin{enumerate}
    \item First let $K=\mathbb{Q}$. Consider the case in Example \ref{unique F example} (1). Let $L=\mathbb{Q}(a)$ and $L'=\mathbb{Q}(a\zeta)$. Thus $LL'=\mathbb{Q}(a,\zeta)=\tilde{L}$. Since the Galois group of Galois closure over $\mathbb{Q}$ is solvable, we are done for any number field $K$ by Lemma \ref{lin disj} (2).

\smallskip
    \item Now $(n,n,n\phi(n))\in \mathcal{C}'_K$ if and only if $\phi(n)\ |\ n$ that is $\Pi_{p|n}(p-1)\ |\ \Pi_{p|n}(p)$. It is easy to see that this holds if and only if $n=1$ or $n=2^l 3^m$ for $l\geq 1$ and $m\geq 0$.\smallskip

    \item Suppose the triplet is irreducible and $n\neq 1$. Let $p|n$ and $n=p^k m$ where $p$ is a prime and $p\nmid m$. Therefore $(n,n,n\phi(n))=(p^k m,p^k m, p^k m\ \phi(p^k m))=(p^k\cdot m, p^k\cdot m, p^k\phi(p^k)\cdot m\phi(m))$. Thus we must have $m=1$. Now $n=p^k$. Thus $(n,n,n\phi(n))=(p^k,p^k,p^k(p^k-p^{k-1}))=(p\cdot p^{k-1}, p\cdot p^{k-1}, p^2\cdot p^{k-1}(p^{k-1}-p^{k-2}))$. Hence we must have $k=1$. The converse is clear.
\end{enumerate}
    
\end{proof}

\begin{remark}
    In particular Proposition \ref{xn comp feas} says that for $p\geq 3$ prime we have $(p,p,p(p-1))\in \mathcal{C}_K$ and we can choose the field extensions such that compositum has Galois closure over $K$ having a solvable Galois group. Also by Proposition \ref{Sn comp feas} we have that for the same triplet $(p,p,p(p-1))\in \mathcal{C}_K$ we can choose the field extensions such that compositum has Galois closure over $K$ having Galois group $\mathfrak{S}_p$ which is non-solvable for $p\geq 5$. 
\end{remark}






\smallskip

\noindent {\it Acknowledgements:}  
The author would like to thank Prof Purusottam Rath, CMI Chennai for suggesting to consider the question of isomorphism classes of extensions with given degree and cluster size which is dealt with in Section \ref{Inverse Root Capacity Problem}. The author is grateful to Prof P Vanchinathan and Dr Anand Chitrao for suggesting to generalize one of the results in \cite{jaiswal2025minimal} for number fields which is achieved in a result in Section \ref{Minimal Generating Sets of Galois Closure}. The author would also like to thank Dr Shripad Garge, IIT Bombay for introducing the paper by Drungilas et al. \cite{drungilas2012degree} and suggesting to investigate the problem of compositum feasible triplets and associated Galois groups which is dealt with in Section \ref{Compositum Feasible Triplets}. Finally the author would like to acknowledge support of IIT Bombay Institute Post Doctoral Fellowship during the later part of this work.  
\medskip

\textbf{Data Availability Statement}:
No new data were created or analyzed in this study.

\medskip

\bibliographystyle{plain}
 \bibliography{mybib}

 \bigskip

\end{document}